\documentclass[11pt]{article}
\usepackage{amsmath, amsfonts, amssymb, amsthm}
\usepackage{natbib}
\DeclareMathSymbol{\leqslant}{\mathalpha}{AMSa}{"36} % nicer `smaller or equal'
\DeclareMathSymbol{\geqslant}{\mathalpha}{AMSa}{"3E} % nicer `larger or equal'
\DeclareMathSymbol{\eset}{\mathalpha}{AMSb}{"3F}     % nicer `emptyset'
\renewcommand{\leq}{\;\leqslant\;}                   % redef. of < or =
\renewcommand{\geq}{\;\geqslant\;}                   % redef. of > or =
\renewcommand{\le}{\leqslant}                   % redef. of < or =
\renewcommand{\ge}{\geqslant}                   % redef. of > or =
\newcommand{\suptwo}[2]{\sup_{\substack{#1 \\ #2}}} % sup with 2 lines
\newcommand{\sumtwo}[2]{\sum_{\substack{#1 \\ #2}}} % sum with 2 lines
\newcommand{\prodtwo}[2]{\prod_{\substack{#1 \\ #2}}}     % product 2 lines
\newcommand{\ind}{1\hspace{-0.098cm}\mathrm{l}}
\newcommand{\heap}[2]{\genfrac{}{}{0pt}{}{#1}{#2}}
\newcommand{\sfrac}[2]{\mbox{$\frac{#1}{#2}$}}
\newcommand{\ssup}[1] {{\scriptscriptstyle{({#1}})}}

\newcommand{\cF}{\mathcal{F}}
\newcommand{\cG}{\mathcal{G}}

\newcommand{\cI}{\mathcal{I}}

\newcommand{\cL}{\mathcal{L}}

\newcommand{\cZ}{\mathcal{Z}}

\newcommand{\ma}{\mathrm{max}}
\newcommand{\mi}{\mathrm{min}}

\renewcommand{\th}{\theta}

\newcommand{\om}{\omega}

\newcommand{\eps}{\varepsilon}
\newcommand{\lam}{\lambda}
\newcommand{\sig}{\sigma}
\newcommand{\vphi}{\varphi}
\newcommand{\vth}{\vartheta}
\newcommand{\Lam}{\Lambda}

\newcommand{\dto}{\downarrow}

\newcommand{\var}{\mathrm{var}}

\newcommand{\IP}{\mathbb{P}}
\newcommand{\IS}{\mathbb{S}}
\newcommand{\IT}{\mathbb{T}}
\newcommand{\II}{\mathbb{I}}
\newcommand{\IIJ}{\mathbb{J}}
\newcommand{\IN}{\mathbb{N}}

\newcommand{\IZ}{\mathbb{Z}}

\newcommand{\IR}{\mathbb{R}}
\newcommand{\IE}{\mathbb{E}}
\newcommand{\iN}{\in\IN}
\newcommand{\iZ}{\in\IZ}
\newcommand{\iR}{\in\IR}

\newcommand{\be}{\begin{eqnarray*}}
\newcommand{\ee}{\end{eqnarray*}}
\newcommand{\ben}{\begin{eqnarray}}
\newcommand{\een}{\end{eqnarray}}

\theoremstyle{plain}
\newtheorem{theo}{Theorem}[section]
\newtheorem{lemma}[theo]{Lemma}
\newtheorem{prop}[theo]{Proposition}
\newtheorem{corollary}[theo]{Corollary}

\theoremstyle{definition}

\newtheorem{remark}[theo]{Remark}

\newtheorem{example}[theo]{Example}

\renewenvironment{proof}[1][] {{\bf Proof#1.} }{\hspace*{\fill}$\square$\medskip\par}

\begin{document}
\vglue20pt \centerline{\huge\bf Random networks with sublinear preferential}
\medskip

\centerline{\huge\bf attachment: Degree evolutions}

\bigskip
\bigskip

\centerline{by}
\bigskip
\medskip

\centerline{{\Large Steffen Dereich$^{1,2}$ and Peter M\"orters$^2$}}
\bigskip

\begin{center}\it
$^1$Institut f\"ur Mathematik, MA 7-5, Fakult\"at II \\
Technische Universit\"at Berlin
\\ Stra\ss e des 17. Juni 136 \\ 10623 Berlin\\
Germany\\[2mm]
%dereich@math.tu-berlin.de\\[2mm]
$^2$Department of Mathematical Sciences\\
University of Bath\\
Claverton Down\\
Bath BA2 7AY\\
United Kingdom\\
%maspm@bath.ac.uk\\

\vspace{1cm}

{\rm Version of \today}\\
\end{center}

\bigskip
\bigskip
\bigskip

{\leftskip=1truecm
\rightskip=1truecm
\baselineskip=15pt
\small

\noindent{\slshape\bfseries Summary.} We define a dynamic model of random networks, where new vertices
are connected to old ones with a probability proportional to a sublinear function of their degree. We first give
a strong limit law for the empirical degree distribution, and then have a closer look at the temporal
evolution of the degrees of individual vertices, which we describe in terms of large and moderate deviation principles.
Using these results, we expose an interesting phase transition: in cases of \emph{strong} preference
of large degrees, eventually a single vertex emerges forever as vertex of maximal
degree, whereas in cases of \emph{weak} preference, the vertex of maximal degree is changing
infinitely often. Loosely speaking, the transition between the two phases occurs in the case when a new edge is
attached to an existing vertex with a probability proportional to the root of its current degree.
\bigskip

\bigskip


}
%%%%%% End of narrower

\newpage

\section{Introduction}\label{intro}

\subsection{Motivation}

Dynamic random graph models, in which new vertices prefer to be attached to
vertices with higher degree in the existing graph, have proved to be immensely
popular in the scientific literature recently. The two main reasons for this popularity
are, on the one hand, that these models can be easily defined and modified, and
can therefore be calibrated to serve as models for social networks, collaboration
and interaction graphs, or the web graph. On the other hand, if the attachment
probability is approximately proportional to the degree of a vertex, the dynamics of the
model can offer a credible explanation for the occurrence of power law degree
distributions in large networks.
\medskip

The philosophy behind these preferential attachment models is that growing networks are built
by adding nodes successively. Whenever a new node is added it is linked  by edges to one or more existing nodes with a probability proportional to a function $f$ of their degree. This function $f$, called \emph{attachment rule}, or sometimes \emph{weight function}, determines the qualitative features of the dynamic network.
\medskip

The heuristic characterisation does not amount to a full definition of the model, and some clarifications have to be made, but it is generally believed that none of these crucially influence the long time behaviour of the model.
\medskip

It is easy to see that in the general framework there are \emph{three} main regimes:

\begin{itemize}
\item the \emph{linear} regime, where $f(k) \asymp k$;
\item the \emph{superlinear} regime, where $f(k) \gg k$;
\item the \emph{sublinear} regime, where $f(k) \ll k$.
\end{itemize}

The linear regime has received most attention, and a major
case  has been introduced in the much-cited paper~\cite{BA99}.
There is by now
a substantial body of rigorous mathematical work on this case. In particular, it is shown
in~\cite{B+01}, \cite{Mo02} that  the empirical degree distribution
follows an asymptotic power law  and in \cite{Mo05} that the maximal
degree of the network is growing polynomially of the same order as the degree of the first node.
\medskip

In the superlinear regime the behaviour is more extreme. In \cite{OS05}
it is shown that  a dominant vertex emerges, which attracts a positive proportion of all future edges. Asymptotically, after $n$ steps, this vertex has degree of order~$n$, while the degrees of all other vertices  are bounded. In the most extreme cases eventually all vertices attach to the dominant vertex. 
\medskip

In the linear and sublinear regimes \cite{RTV07} find almost sure convergence of the empirical degree distributions. In the linear regime the limiting distribution obeys a power law, whereas in the sublinear regime  the limiting distributions are stretched exponential distributions.
Apart from this, there has not been much research so far in the sublinear regime, which is the main
concern of the present article, though we include the linear regime in most
of our results.

\medskip
\pagebreak

{Specifically, we discuss a preferential attachment model}
where new nodes connect to a random number of old nodes, which in
fact is quite desirable from the modelling point of view.  More
precisely, the node added in the $n$th step is connected
independently to any old one with probability $f(k)/n$, where $k$ is
the (in-)degree of the old node. We first determine the asymptotic
degree distribution, see Theorem~\ref{EDD}, and find a result which
is in line with that of \cite{RTV07}. The result implies in
particular that, if $f(k)=(k+1)^\alpha$ for~$0\leq\alpha<1$, then the
asymptotic degree distribution~$(\mu_k)$ satisfies
$$\log \mu_k \sim - \sfrac1{1-\alpha}\, k^{1-\alpha},$$
showing that power law behaviour is limited to the linear regime.
Under the
assumption that the strength of the attachment preference is sufficiently weak, we give
very fine results about the probability that the degree of a fixed vertex follows
a given increasing function, see Theorem~\ref{LDP} and Theorem~\ref{MDP}.
These large and moderate deviation results, besides being of independent interest,
play an important role in the proof of our main~result.
This result describes an interesting dichotomy about the behaviour of the vertex of
maximal degree, see Theorem~\ref{main}:
\begin{itemize}
\item \emph{The strong preference case:} If $\sum_n  1/f(n)^2< \infty$, then there exists a single dominant vertex --called \emph{persistent hub}-- which has maximal
degree for all but finitely many times. However, only in the linear regime the number of
new vertices connecting to the dominant vertex is growing polynomially in time.
\item \emph{The weak preference case:} If $\sum_n  1/f(n)^2= \infty$, then there is almost surely \emph{no} persistent hub.
In particular, the index, or time of birth, of the current vertex of maximal degree
is a function of time diverging to infinity in probability.
In Theorem~\ref{WL} we provide asymptotic results
for the index and degree of this vertex, as time goes to infinity.
\end{itemize}
A rigorous definition of the model is given in Section~\ref{12}, and precise statements of all
the principal results follow in Section~\ref{13}. At the end of that section, we also give a
short overview over the further parts of this paper.
\medskip

\subsection{Definition of the model}\label{12}

We now explain how precisely we define our preferential attachment model given a
monotonically increasing \emph{attachment rule} $f \colon \{0,1,2,\ldots\}
\longrightarrow (0,\infty)$ with $f(n)\leq n+1$ for all $n\in\IZ_+:=\{0,1,\ldots\}$.
At time $n=1$ the network consists of a single vertex (labeled $1$) without edges and for each $n\iN$  the graph evolves in the time step $n\to n+1$ according to the following rule
\begin{itemize}
\item add a new vertex (labeled $n+1$) and
\item insert for each old vertex $m$ a directed edge $n+1\to m$ with probability
    $$ \frac{f(\text{indegree of }m\text{ at time }n)}{n}.$$
\end{itemize}
The new edges are inserted independently for each old vertex. Note that the assumptions
imposed on $f$ guarantee that in each evolution step the probability for adding an edge
is smaller or equal to $1$. Formally we are dealing with a directed network, but indeed,
by construction, all edges are pointing from the younger to the older vertex, so that the
directions can trivially be recreated from the undirected (labeled) graph.
\smallskip
%\pagebreak

There is one notable change to the recipe given in~\cite{KR00}: We do not add one edge in
every step but a random number, a property which is actually desirable in most applications.
Given the graph after attachment of the $n$th vertex, the expected number of edges added in the next  step is
$$\frac 1n \sum_{m=1}^n f\big(\text{indegree of } m \text{ at time }n\big)\, .$$
This quantity converges, as $n\to\infty$ almost surely to a
deterministic limit $\lam$, see Theorem~\ref{EDD}. Moreover, the law
of the number of edges added is asymptotically Poissonian with
parameter~$\lam$. Observe that the \emph{out}degree of every vertex
remains unchanged after the step in which the vertex was created.
Hence our principal interest when studying the asymptotic evolution
of degree distributions is in the \emph{in}degrees.
\smallskip

%\pagebreak

\subsection{Presentation of the main results}\label{13}

We denote by $\cZ[m,n]$, for $m,n\iN$, $m\leq n$, the indegree of the $m$-th vertex after
the insertion of the $n$-th vertex, and by $X_k(n)$ the proportion of nodes of
indegree $k\iZ_+$ at time $n$, that is
$$X_k(n) = \frac1n \sum_{i=1}^n \ind_{\{\cZ[i,n]=k\}}.$$
Moreover,  denote $\mu_k(n)=\IE X_k(n)$,  $X(n)=(X_k(n) \colon k\iZ_+)$,
and $\mu(n)=(\mu_k(n) \colon k\iZ_+)$.\smallskip

\begin{theo}[Asymptotic empirical degree distribution]\label{EDD} \
\begin{enumerate}
\item[(i)] Let
$$\mu_k=\frac1{1+f(k)} \prod_{l=0}^{k-1} \frac{f(l)}{1+f(l)} \qquad \mbox{ for $k\iZ_+$,}$$
which is a sequence of probability weights. Then, almost surely,
$$\lim_{n\to\infty} X(n)= \mu$$
in total variation norm.

\item[(ii)] If $f$ satisfies $f(k)\leq \eta k+1$ for some $\eta\in(0,1)$, then  the conditional
distribution of the outdegree of the $(n+1)$st incoming node (given
the graph at time $n$)  converges almost surely in variation
topology to the Poisson distribution with parameter
${\lambda:}=\langle\mu, f\rangle$.

\end{enumerate}
\end{theo}

\begin{remark} The asymptotic degree distribution coincides with
that in the random tree model introduced in \cite{KR00} and studied
by \cite{RTV07}, if $f$ is chosen as an appropriate multiple of
their weight function. This is strong evidence that these models
show the same qualitative behaviour, {and that our further
results hold \emph{mutatis mutandis} for preferential attachment
models in which new vertices connect to a fixed number of old ones.}
\end{remark}

\begin{example}\label{ex_deg_reg}
Suppose ${f(k)\sim \gamma k^\alpha}$, for $0<\alpha<1$  and $\gamma>0$, then
a straight forward analysis yields that
$$
\log \mu_k \sim - \sum_{l=1}^{k+1} \log\big(1+ l^{-\alpha} \big)
\sim - \sfrac1\gamma\, \sfrac1{1-\alpha} \, k^{1-\alpha}\, .
$$
Hence the asymptotic degree distribution has stretched exponential tails.
\end{example}
\smallskip

In order to analyse the network further, we scale the time
as well as the way of counting the indegree. To the \emph{original}
time $n\iN$ we associate an \emph{artificial} time
$$\Psi(n):=\sum_{m=1}^{n-1} \frac1m \sim \log n,$$ and to the
\emph{original} degree $j\iZ_+$ we associate the \emph{artificial}
degree
$$\Phi(j):=\sum_{k=0}^{j-1}\frac1{f(k)}.$$
An easy law of large numbers illustrates the role of these scalings.

\begin{prop}[Law of large numbers]\label{LLN}
For any fixed vertex labeled $m\in\IN$, we have that
$$ \lim_{n\to\infty} \frac{\Phi(\cZ[m,n])}{\Psi(n)} =1  \qquad \mbox{ almost surely}\, .$$
\end{prop}

\begin{remark}\label{ReLLN}
Since $\Psi(n)\sim \log n$, we conclude that for any $m\iN$, almost surely,
$$
\Phi(\cZ[m,n])\sim \log n \text{ as }n\to\infty.
$$
In particular, we get for an attachment rule $f$ with   $f(n)\sim \gamma n$ and $\gamma\in(0,1]$, that $\Phi(n)\sim \frac 1{\gamma} \log n$ which implies that
$$
\log \cZ[m,n] \sim \log n^{\gamma}\text{, almost surely.}
$$
{Furthermore, an attachment rule with $f(n)\sim \gamma n^\alpha$ for $\alpha<1$ and $\gamma>0$ leads to
$$
\cZ[m,n]\sim(\gamma (1-\alpha) \log n)^{\frac1{1-\alpha}}.
$$}
\end{remark}
%\medskip

We denote by $\IT:=\{\Psi(n)\colon n\iN\}$ the set of artificial
times, and by $\IS:=\{\Phi(j):j\iZ_+\}$ the set of artificial
degrees. From now on, we refer by \emph{time} to the artificial
time, and by \emph{(in-)degree} to the artificial degree. Further,
we introduce a new real-valued process $(Z[s,t])_{s\in\IT, t \ge
0}$ via
$$Z[s,t]:=\Phi(\cZ[m,n]) \quad \mbox{if $s=\Psi(m)$, $t=\Psi(n)$ and $m\leq n$,}$$
and extend the definition to arbitrary $t$  by  letting
$Z[s,t]:=Z[s,s\vee \max (\IT\cap [0,t])]$. For notational convenience we extend the definition
of $f$ to $[0,\infty)$ by setting $f(u):=f(\lfloor u\rfloor)$ for all $u\in[0,\infty)$,
and linearly interpolate $\Phi$ at the breakpoints $\IS$ so that
$$\Phi(u)=\int_{0}^u \frac1{f(v)}\,dv.$$
We denote by ${\mathcal L}[0,\infty)$ the space of c\`adl\`ag functions  $x\colon
[0,\infty)\to\IR$ endowed with the topology of uniform convergence
on compact subsets of~$[0,\infty)$.

\begin{prop}[Central limit theorem]\label{CLT}
In the case of weak preference, for all $s\in\IT$,
%a central limit theorem holds, i.e.
$$\Big(\frac{Z[s,s+\vphi^*_{\kappa t}]-\vphi^*_{\kappa t}}{\sqrt\kappa} \colon t \geq 0 \Big)
\Rightarrow (W_t \colon t\geq0 )\, ,$$ in distribution on ${\mathcal L}[0,\infty)$, 
where $(W_t \colon t\geq 0)$ is a standard Brownian motion
and $(\vphi^*_t)_{t\ge0}$ is the inverse of $(\vphi_t)_{t\ge0}$
given by
$$\vphi_t=\int_0^{\Phi^{-1}(t)}\frac1{f(u)^2}\,du .$$
\end{prop}

Our \emph{main} result describes the behaviour of the vertex
of \emph{maximal} degree, and reveals an interesting dichotomy between weak and strong forms
of preferential attachment. %phase transition.
\smallskip

\begin{theo}[Vertex of maximal degree]\label{main}
Suppose $f$ is concave. Then we have the following dichotomy:
\begin{description}
\item[Strong preference.] If
$$
\sum_{k=0}^\infty  \frac1{f(k)^2}<\infty,
$$
then with probability one there exists a persistent hub, i.e.\ there is a single vertex
which has maximal indegree for all but finitely many times.
\item[Weak preference.] If
$$
\sum_{k=0}^\infty  \frac1{f(k)^2}=\infty,
$$
%then with probability one there exists no persistent hub.
then with probability one there exists no persistent hub and
the time of birth, or index, of the current hub tends to infinity in
probability.
\end{description}
\end{theo}

\begin{remark}{Without the assumption of concavity of $f$, the assertion remains true in the weak preference
regime. In the strong preference regime our results still imply that, almost surely, the number of
vertices, which at some time have maximal indegree, is finite.}
\end{remark}

\begin{remark}
In the weak preference case the information about the order of the vertices is asymptotically lost:
as a consequence of the proof of Theorem \ref{main}, we have for two nodes $s<s'$ in $\IT$ that
$$\lim_{t\to\infty} \IP(Z[s,t]>Z[s',t])=\sfrac12,$$
a phenomenon reminiscent of \emph{propagation of chaos}.
Conversely, in the strong preference case, the information about the
order is not lost completely and one has
$$\lim_{t\to\infty} \IP(Z[s,t]>Z[s',t])>\sfrac12.$$
\end{remark}
\smallskip

Investigations so far were centred around \emph{typical}
vertices in the network. Large deviation principles, as provided
below, are the main tool to analyse \emph{exceptional} vertices in
the random network. Throughout we use the large-deviation
terminology of \cite{DemZei98} and, from this point on, the focus is
on the weak preference case.
\smallskip

Our aim is to determine the typical age and indegree evolution of
the hub. For this purpose we assume that
\begin{equation}\label{ass}
\begin{aligned}
\bullet & \mbox{ $f$ is regularly varying with index~$0\leq \alpha<\sfrac12$,} \\[1mm]
\bullet & \mbox{ for some~$\eta<1$, we have $f(j)\leq \eta(j+1)$ for all $j\in\IZ_+$.}\hspace{5cm}
\end{aligned}
\end{equation}
 We set $\bar f:=f\circ\Phi^{-1},$ and recall 
from Lemma \ref{reglema} in the appendix 
that we can represent $\bar f$ as $\bar f(u)=u^{\alpha/(1-\alpha)} 
\,\bar \ell(u)$ for $u>0$, where $\bar \ell$ is a slowly varying function.
We denote by ${\mathcal I}[0,\infty)$ the space of nondecreasing
functions $x\colon [0,\infty)\to\IR$ with $x(0)=0$
endowed with the topology of uniform convergence on compact 
subintervals of $[0,\infty)$.

\begin{theo}[Large deviation principles]\label{LDP}
Under assumption~\eqref{ass},
for every $s\in\IT$, the family of functions
$$\Big(\frac1{\kappa} Z[s,s+\kappa t] \colon t\geq 0 \Big)_{\kappa>0}$$
satisfies large deviation principles on the space ${\mathcal I}[0,\infty)$,
\begin{itemize}
\item with speed $(\kappa^{\frac1{1-\alpha}}\bar \ell(\kappa))$ and good rate function
$$J(x) =  \begin{cases}
\int_{0}^\infty x_t^{\frac{\alpha}{1-\alpha}} [1-\dot x_t+\dot x_t \log \dot x_t]\,dt&\text{ if }x\text{ is absolutely continuous,}\\
\infty &\text{ otherwise}.
\end{cases}$$
\item and with speed $(\kappa)$ and good rate function
$$K(x) =  \begin{cases} a\, f(0)
&\text{ if }x_t=(t-a)_+ \text{ for some }a\geq 0,\\
\infty &\text{ otherwise}.
\end{cases}$$
\end{itemize}
\end{theo}

\begin{remark} The large deviation principle states, in particular, that the most likely deviation
from the growth behaviour in the law of large numbers is having zero indegree for some (unusually
long) time, and after that time typical behaviour kicking in.
\end{remark}

More important for our purpose is a moderate deviation principle,
which describes deviations on a finer scale. Similar as before, 
we denote by ${\mathcal L}(0,\infty)$ the space of c\`adl\`ag functions  $x\colon
(0,\infty)\to\IR$ endowed with the topology of uniform convergence
on compact subsets of~$(0,\infty)$, and always use the convention
$x_0:=\liminf_{t\downarrow 0} x_t$.

\begin{theo}[Moderate deviation principle]\label{MDP}
Suppose \eqref{ass} and that $(a_\kappa)$ is regularly varying, so that the limit
$$c:= \lim_{\kappa\uparrow\infty}
a_\kappa\, \kappa^{\frac{2\alpha-1}{1-\alpha}} \,
\bar{\ell}(\kappa) \in[0,\infty)$$ exists.
If
$\kappa^{\frac{1-2\alpha}{2-2\alpha}} \,
\bar{\ell}(\kappa)^{-\frac12} \ll a_\kappa \ll \kappa,$ then, for
any  $s\in\IT$,  the family of functions
$$\Big( \frac{Z[s,s+\kappa t] - \kappa t}{a_\kappa} \colon t \geq 0 \Big)_{\kappa>0}$$
satisfies a large deviation principle on ${\mathcal L}(0,\infty)$
with speed $(a_\kappa^2 \, \kappa^{\frac{2\alpha-1}{1-\alpha}} \,
\bar{\ell}(\kappa))$ and good rate function
$$I(x)= \left\{ \begin{array}{ll}\frac 12  \int_0^\infty (\dot{x}_t)^2 \, t^{\frac{\alpha}{1-\alpha}} \, dt
- \sfrac1c\,f(0)\,x_0 & \mbox{ if $x$ is absolutely continuous and $x_0\leq 0$,}\\
\infty & \mbox{ otherwise,} \end{array} \right.$$
where we use the convention $1/0=\infty$.
\end{theo}

\begin{remark}{If $c=\infty$ there is still a moderate deviation principle
on the space of functions $x\colon (0,\infty) \to \IR$ with the
topology of pointwise convergence. However, the rate function $I$,
which has the same form as above with $1/\infty$ interpreted as
zero, fails to be a good rate function.}
\end{remark}

\begin{remark}\label{CLT2} Under assumption~\eqref{ass} the 
central limit theorem of Proposition~\ref{CLT} can be stated
as a complement to the moderate deviation principle: For $a_\kappa\sim
\kappa^{\frac{1-2\alpha}{2-2\alpha}} \bar\ell(\kappa)^{-\frac12}$,
we have
$$\left(\frac{Z[s,s+\kappa t]-\kappa t}{a_\kappa} \colon t\geq0
\right)\Rightarrow \Bigl( \sqrt{\sfrac{1-\alpha}{1-2\alpha}}\,
W_{t^{\frac{1-2\alpha}{1-\alpha}}}\colon t\geq 0\Bigr).$$ 
See Section~\ref{sec2.1} for details.
\end{remark}
%\bigskip

{Our final result describes weak limit laws for index and
degree of the vertex of maximal degree. This result relies on the
moderate deviation principle above.}
\smallskip

\begin{theo} [Limit law for age and degree of the vertex of maximal degree]\label{WL}
Suppose~$f$ is regularly varying with index $\alpha<\frac12$. Define
$s^*_t$ to be the index of the hub at time~$t$, and
$Z^{\max}_t=Z[s^*_t,t]$ to be the corresponding maximal indegree.
One has, in probability,
$$
s_t^* \sim Z_t^{\max} - t \sim  \frac12 \frac{1-\alpha}{1-2\alpha}\frac{t^{\frac{1-2\alpha}{1-\alpha}}}{\bar \ell(t)}=\frac12 \frac{1-\alpha}{1-2\alpha}\frac{t}{\bar f(t)}.
$$
Moreover, in probability on $\cL(0,\infty)$,
$$\lim_{t\to\infty}
\Big(\frac{Z[s^*_t,s^*_t+tu]-tu}{t^{\frac{1-2\alpha}{1-\alpha}} \bar\ell(t)^{-1}} \colon u\geq 0 \Big)
=\Big(\sfrac{1-\alpha}{1-2\alpha} \bigl(u^{\frac{1-2\alpha}{1-\alpha}}\wedge 1\bigr)\colon u\geq 0\Big).
$$
\end{theo}

\begin{remark}
In terms of the natural scaling, we get for the index $m^*_n$ of the hub and the maximal indegree $\cZ^{\max}_n$ at natural time $n\iN$ 
that, in probability,
$$
\log m_n^*\sim \frac12 \frac{1-\alpha}{1-2\alpha}\frac{(\log n)^{\frac{1-2\alpha}{1-\alpha}}}{\bar \ell(\log n)}
$$
and
$$
\cZ^{\max}_n- \Phi^{-1}(\Psi( n))\sim \frac12 \frac{1-\alpha}{1-2\alpha} \log n.
$$
\end{remark}

\medskip

The remainder of this paper is devoted to the proofs of these results. Rather than proving
the results in the order in which they are stated, we proceed by the techniques
used. Section~\ref{sec3} is devoted to martingale techniques, which in particular prove the
law of large numbers, Proposition~\ref{LLN}, and the central limit theorem, Proposition~\ref{CLT}.
We also prove a property of the martingale limit %(Proposition \ref{prop_ac}) 
which is crucial in the proof of Theorem~\ref{main}. Section~\ref{sec4} is using Markov chain techniques and provides the proof of Theorem~\ref{EDD}.
In Section~\ref{sec5} we collect the large deviation techniques, proving Theorem~\ref{LDP} and
Theorem~\ref{MDP}. Section~\ref{sec6} combines the various techniques to prove our main result,
Theorem~\ref{main}, along with Theorem~\ref{WL}. An appendix collects the auxiliary statements from
the theory of regular variation.
\bigskip

\section{Martingale techniques}\label{sec3}

In this section we identify a martingale associated with the degree evolution
of a vertex, and study its properties. This will be a vital tool in the further
analysis of the network.

\subsection{Martingale convergence}\label{sec2.1}

\begin{lemma}\label{degs}
Fix $s\in\IT$ and represent $Z[s,\cdot\,]$ as
$$
Z[s,t]= t-s +M_t.
$$
Then $(M_{t})_{t\in \IT, t\geq s}$ is a martingale. Moreover,
the martingale converges if and only if
\begin{align}\label{eq0727-1}
\int_0^\infty \frac1{f(u)^2}\,du<\infty,
\end{align}
and otherwise  it satisfies the following functional
central limit theorem: Let
% in the Skorokhod topology: extend  the definition of $M$ by setting $M_t=M_{\max \IT\cap [0,t]}$
% for $t\in[0,\infty)$, let
$$\vphi_{t}:=\int_0^{\Phi^{-1}(t)} \frac {1}{f(v)^2}\,dv=\int_0^t \frac1{\bar f(v)}\, dv,$$
and denote by $\vphi^*\colon[0,\infty)\to[0,\infty)$ the inverse of $(\vphi_t)$; then the martingales $$M^\kappa:=\left(\frac1{\sqrt \kappa} M_{s+\vphi^*_{\kappa t}}\right)_{t\ge 0}\qquad \text{ for }\kappa>0 $$
converge in distribution to standard Brownian motion as $\kappa$ tends to infinity.
In any case the processes $(\frac1\kappa\, Z[s,s+\kappa t])_{t\ge 0}$ converge, as~$\kappa\uparrow\infty$,
almost surely, in ${\mathcal L}[0,\infty)$ to the identity.
\end{lemma}

\begin{proof}
For $t=\Psi(n)\in\IT$ we denote by $\Delta t$ the distance between $t$ and its right neighbour in $\IT$, i.e.
\begin{align}\label{delta}
\Delta t= \frac1n =\frac{1}{\Psi^{-1}(t)}.
\end{align}
One has
\begin{align*}%\label{eq0726-1}
\IE \big[ Z[s,t+\Delta t]-Z[s,t]\,\big|\,\cG_n\big] & =
\IE \big[ \Phi\circ\cZ[i,n+1]-\Phi\circ\cZ[i,n]\,\big|\,\cG_n\big]\notag\\[2mm] & =
\frac{f(\cZ[i,n])}{n} \times \frac{1}{f(\cZ[i,n])}
= \frac1{n}.
\end{align*}
Moreover,
\begin{align}\label{eq0726-2}
\langle M\rangle_{t+\Delta t} -\langle M\rangle_{t} =  \Bigl(1-  \bar f(Z[s,t]) \Delta t\Bigr) \frac{ 1}{\bar f(Z[s,t])}   \Delta t\leq \frac{ 1}{\bar f(Z[s,t])} \Delta t.
\end{align}

Observe that by Doob's martingale inequality and the uniform boundedness of $\bar f(\cdot)^{-1}$ one has
$$
a_i:= \frac{\IE[ \sup_{s\leq t \leq s+2^{i+1}} |M_t|^2]}{\bigl(2^{i/2} \log 2^i\bigr)^2}\leq C\, \frac1{i^2},
$$
where $C=C(f(0))$ is a constant only depending on $f(0)$. Moreover, by Chebyshev, one has
\begin{align*}\begin{split}%\label{eq1009-1}
\IP & \Bigl(\sup_{t\geq s+2^i} \frac{|M_t|}{\sqrt{t-s}
\log(t-s)}\ge1\Bigr) \leq \sum_{k=i}^\infty \IP \Bigl(\sup_{s+2^k
\le t\le s+2^{k+1}}
\frac{M_t^2}{(t-s) \log^2(t-s)}\ge1\Bigr) \\
& \leq \sum_{k=i}^\infty \IE \Bigl[\sup_{s+2^k \le t\le s+2^{k+1}}
\frac{M_t^2}{(t-s) \log^2(t-s)}\Bigr]
\leq \sum_{k=i}^\infty a_k <\infty.
\end{split}\end{align*}
Letting $i$ tend to infinity, we conclude that almost surely
\begin{align}\label{eq0727-3}
\limsup_{t\to\infty}\frac{|M_t|}{\sqrt{t-s} \log(t-s)}\leq1.
\end{align}
In particular, we obtain almost sure convergence of $(\frac1\kappa Z[s,s+\kappa t])_{t\geq 0}$ to the identity. % in the Skorokhod topology.
As a consequence of (\ref{eq0726-2}), for any $\eps>0$, there exists a random almost surely finite constant $\eta=\eta(\om,\eps)$ such that,
for all~$t\geq s$,
$$\langle M\rangle_t\leq  \int_0^{t-s}  \frac1{f(\Phi^{-1}((1-\eps) u))}\,du +\eta.$$
Note that $\Phi:[0,\infty)\to[0,\infty)$ is bijective and substituting $(1-\eps)\kappa u$ by $\Phi(v)$ leads to
$$
\langle M\rangle_t\leq  \frac{1}{1-\eps} \int_0^{\Phi^{-1}((1-\eps)(t-s))}   \frac1{f(v)^2}\,dv+\eta\leq  \frac{1}{1-\eps} \int_0^{\Phi^{-1}(t-s)}   \frac1{f(v)^2}\,dv+\eta.
$$
Thus,  condition (\ref{eq0727-1}) implies convergence of the martingale $(M_t)$.

We now assume that $(\vphi_t)_{t\geq0}$ converges to infinity. Since $\eps>0$ was arbitrary the above estimate implies that
$$
\limsup_{t\to\infty} \frac{\langle M \rangle_t}{\vphi_{t-s}}\leq 1,  \text{ almost surely.}
$$
To conclude the converse estimate note that $\sum_{t\in\IT} (\Delta t)^2 <\infty$ so that we get with
 (\ref{eq0726-2}) and (\ref{eq0727-3}) that
$$
\langle M\rangle_t\geq  \int_0^{t-s}  \frac1{f(\Phi^{-1}((1+\eps) u))}\,du -\eta\geq \frac{1}{1+\eps} \int_0^{\Phi^{-1}(t-s)}   \frac1{f(v)^2}\,dv-\eta,
$$
for an appropriate finite random variable $\eta$. Therefore,
\begin{align}\label{eq0827-1}
\lim_{t\to\infty} \frac{\langle M\rangle_t}{\vphi_{t-s}}=1 \qquad\mbox{almost surely.}
\end{align}
The jumps of $M^\kappa$ are uniformly bounded by a deterministic value that tends to zero as $\kappa$ tends to $\infty$. By a functional central limit theorem for martingales (see, e.g., Theorem 3.11 in \cite{JacShi03}),
the central limit theorem follows once we establish that, for any $t\geq0$,
$$\lim_{\kappa\to\infty} \langle M^\kappa\rangle _t =t \quad \mbox{ in probability},$$
which is an immediate consequence of (\ref{eq0827-1}).
%(see \cite{HaHey80}, Theorem 3.2).
\end{proof}

\label{CLT2proof}
\begin{proof}[ of Remark~\ref{CLT2}]
We suppose that $f$ is regularly varying with index
$\alpha<\frac12$.   By the central limit theorem
the processes
$$ (Y_t^\kappa:t\geq 0):= \Bigl(\frac{Z[s,s+\vphi^*_{t \vphi_\kappa}]-\vphi^*_{t \vphi_\kappa}}{\sqrt{\vphi(\kappa)}}:t\geq 0\Bigr)\qquad
\mbox{ for }\kappa>0$$
converge in distribution to the Wiener process $(W_t)$ as $\kappa$ tends to infinity.
For each $\kappa>0$ we consider the time change $(\tau^\kappa_t)_{t\ge 0}:=(\vphi_{\kappa t}/\vphi_\kappa)$. Using that $\vphi$ is regularly varying with parameter $\frac{1-2\alpha}{1-\alpha}$ we find uniform convergence on compacts:
$$
(\tau^\kappa_t) \to  (t^{\frac{1-2\alpha}{1-\alpha}}) =:(\tau_t)\qquad \text{as }\kappa\to\infty.
$$
Therefore,
$$
\Bigl(\frac{Z[s,s+\kappa t]-\kappa t}{\sqrt{\vphi(\kappa)}}:t\geq 0\Bigr)= ( Y_{\tau_t^\kappa}^\kappa:t\geq 0) \Rightarrow (W_{\tau_t}:t\geq0).
$$
and, as shown in Lemma \ref{reglema}, $\vphi_\kappa\sim\frac{1-\alpha}{1-2\alpha} \kappa^{\frac{1-2\alpha}{1-\alpha}} \bar \ell(\kappa)^{-1}$.
\end{proof}

\subsection{Absolute continuity of the law of $M_\infty$}

In the sequel, we consider the martingale $(M_t)_{t\ge s, t\in\IT}$ given by $Z[s,t]-(t-s)$ for a fixed $s\in\IT$ in the case of \emph{strong} preference. We denote by $M_\infty$ the limit of the martingale.

\begin{prop}\label{prop_ac}
If $f$ is concave, then the distribution of $M_\infty$ is absolutely continuous with respect to Lebesgue measure.
\end{prop}

\begin{proof}
For ease of notation, we denote $Y_t=Z[s,t]$, for $t\in\IT$, $t\geq s$. 
Moreover, we fix $c>0$ and let $A_t$ denote the event that $Y_u\in [u-c,u+c]$ for all $u\in[s,t]\cap \IT$.

Now observe that for two neighbours $v_-$ and $v$ in $\IS$
\begin{align}\label{eq2401-1}
\IP( Y_{t+\Delta t}=v; A_t)=  (1-\bar f(v)\,\Delta t) \,\IP( Y_{t}=v; A_t) + \bar f(v_-) \, \Delta t \, \IP( Y_{t}=v_-; A_t).
\end{align}
Again we use the notation $\Delta t=\frac1{\Psi^{-1}(t)}$. Moreover, we denote $\Delta \bar f(v)=\bar f(v)- \bar f(v_-)$. In the first step of the proof we derive an upper bound for
$$
h(t)=\max_{v\in\IS} \IP(Y_t=v;A_t)\qquad \mbox{ for }t\in\IT, t\geq s.
$$
With (\ref{eq2401-1}) we conclude that
$$
\IP( Y_{t+\Delta t}=v; A_t) \leq (1- \Delta \bar f(v)\,\Delta t) \, h(t).
$$
For $w\geq 0$ we denote $\varsigma(w)=\max \IS\cap [0,w]$. Due to the concavity of $f$, we get that
$$
h(t+\Delta t)\leq (1-\Delta \bar f(\varsigma(t+c+1))\,\Delta t)\, h(t).
$$
Consequently,
$$
h(t)\leq \prod_{u\in [s,t)\cap \IT} (1-\Delta\bar f(\varsigma(u+c+1)) \,\Delta u)
$$
and using that $\log (1+x)\leq x$ we obtain
\begin{align}\label{2401-2}
h(t)\leq \exp\Bigl( -\sum_{u\in [s,t)\cap \IT} \Delta\bar f(\varsigma(u+c+1)) \,\Delta u\Bigr).
\end{align}
We continue with estimating the sum $\Sigma$ in the latter exponential:
\begin{align*}
\Sigma= \sum_{u\in [s,t)\cap \IT} \Delta\bar f(\varsigma(u+c+1)) \,\Delta u\geq \int_s^t \Delta\bar f(\varsigma(u+c+1)) \,du.
\end{align*}
Next, we denote by $f^\mathrm{lin}$ the continuous piecewise linear interpolation of $f|_{\IN_0}$. Analogously, we set $\Phi^\mathrm{lin}(v)=\int_0^v \frac 1{f^\mathrm{lin}(u)}\, du$ and $\bar f^\mathrm{lin}(v)=f^\mathrm{lin}\circ (\Phi^\mathrm{lin})^{-1}(v)$. Using again the concavity of $f$ we conclude that
$$
\int_s^t \Delta\bar f(\varsigma(u+c+1)) \,du \geq  \int_s^t (f^\mathrm{lin})'(\Phi^{-1}(u+c+1)) \,du,
$$
and that
$$
f^\mathrm{lin}\geq f \ \Rightarrow\ \Phi^\mathrm{lin}\leq \Phi\ \Rightarrow\ (\Phi^\mathrm{lin})^{-1}\geq \Phi^{-1}
\ \Rightarrow \ (f^\mathrm{lin})'\circ (\Phi^\mathrm{lin})^{-1} \leq (f^\mathrm{lin})'\circ \Phi^{-1}.
$$
Hence,
$$
\Sigma\geq \int_s^t (f^\mathrm{lin})'(\Phi^{-1}(u+c+1)) \,du \geq \int_s^t (f^\mathrm{lin})'\circ (\Phi^\mathrm{lin})^{-1}(u+c+1) \,du.
$$
For Lebesgue almost all arguments one has
$$(\bar f^\mathrm{lin})'=(f^\mathrm{lin}\circ (\Phi^\mathrm{lin})^{-1})'= (f^\mathrm{lin})'\circ (\Phi^\mathrm{lin})^{-1} \cdot ((\Phi^\mathrm{lin})^{-1})'=(f^\mathrm{lin})'\circ (\Phi^\mathrm{lin})^{-1} \cdot f^\mathrm{lin}  \circ (\Phi^\mathrm{lin})^{-1}
$$
so that
$$
(f^\mathrm{lin})'\circ (\Phi^\mathrm{lin})^{-1}= \frac{(\bar f^\mathrm{lin})'}
{\bar f^\mathrm{lin}}= (\log   \bar f^\mathrm{lin})'.
$$
Consequently,
$$
\Sigma\geq \log   \bar f^\mathrm{lin}(t+c+1)  -\log   \bar f^\mathrm{lin}(s+c+1)
$$
Using that $f^\mathrm{lin}\geq f$ and $(\Phi^\mathrm{lin})^{-1}\geq \Phi^{-1}$ we finally get that
$$
\Sigma \geq \log   \bar f(t+c+1) -\log c^*,
$$
where $c^*$ is a positive constant not depending on $t$.
Plugging this estimate into (\ref{eq2401-1}) we get
$$
h(t)\leq \frac {c^*}{\bar f(t+c+1)}.
$$

Fix now an interval $I\subset \IR$ of finite length and note that
$$
\IP(M_t\in I;A_t)= \IP(Y_t\in t-s+I;A_t) \leq \# [(t-s+I)\cap\IS\cap A_t]  \cdot h(t).
$$
Now $(t-s+I)\cap\IS\cap A_t$ is a subset of $[t-c,t+c]$ and the minimal distance of two distinct elements is bigger than $\frac 1{\bar f(t+c)}$. Therefore, $\# [(t-s+I)\cap\IS\cap A_t]\leq |I| \,\bar f(t+c)+1$, and
$$
\IP(M_t\in I;A_t)\leq  c^*\,|I| +\frac{c^*}{\bar f(t+c)}.
$$
Moreover, for any open and thus immediately also for any arbitrary interval $I$ one has
$$
\IP(M_\infty\in I;A_\infty)\leq \liminf_{t\to\infty} \IP(M_t\in I;A_t)\leq c^* \, |I|,
$$
where $A_\infty=\bigcap_{t\in[s,\infty)\cap \IT} A_t$.
Consequently, the Borel measure $\mu_c$ on $\IR$ given by $\mu_c(E) = \IE[\ind_{A_\infty} \cdot \ind_{E}(M_\infty)]$,   is absolutely continuous with respect to Lebesgue measure. The distribution $\mu$ of $M_\infty$, i.e.\ $\mu(E) = \IE[\ind_E(M_\infty)]$, can be written as monotone limit of the absolutely continuous measures $\mu_c$ $(c\iN)$, and it is thus also absolutely continuous.
\end{proof}

\section{The empirical indegree distribution}\label{sec4}

In this section we prove Theorem~\ref{EDD}. For $k\iZ_+$ and $n\iN$ let  $\mu_k(n)=\IE [X_k(n)]$ and $\mu(n)=(\mu_k(n))_{k\iZ_+}$. We first
show that $(\mu(n))_{n\iN}$ converges to $\mu=(\mu_k)_{k\iZ_+}$
as $n$ tends to infinity. We start by deriving a recursive representation for $\mu(n)$. For $k\iZ_+$,
\begin{align*}
\IE [X_k(n+1)|X(n)]&= \frac 1{n+1}  \Bigl( \sum_{i=1}^n \IE\Bigl[-\ind_{\{\cZ[i,n]=k<\cZ[i,n+1]\}}+\ind_{\{\cZ[i,n]<k=\cZ[i,n+1]\}}\Big|X(n) \Bigr]\\
& \hspace{4cm}  +n X_k(n)+  \ind_{\{k=0\}}\Bigr)\\
& =X_k(n)+\frac1{n+1} \Bigl[-n X_k(n) \frac{f(k)}n +n X_{k-1}(n) \frac{f(k-1)}n -X_k(n)+\ind_{\{k=0\}}\Bigr].
\end{align*}
Thus the linearity and the tower property of conditional expectation gives
$$\mu_k(n+1) = \mu_k(n) +\frac1{n+1} (f(k-1)\mu_{k-1}(n) - (1+f(k)) \mu_k(t)+ \ind_{\{k=0\}}).$$
Now defining $Q\in \IR^{\IN\times \IN}$ as
\begin{align}\label{defQ}Q=\begin{pmatrix}
-f(0)& f(0) \\
1&-(f(1)+1)&f(1)\\
1& &-(f(2)+1)&f(2)\\
\vdots&&& \ddots& \ddots
\end{pmatrix}
\end{align}
and conceiving $\mu(n)$ as a row vector we can rewrite the recursive equation as
$$\mu(n+1)= \mu(n) \Bigl(I+ \frac1{n+1}\, Q \Bigr),$$
where $I=(\delta_{i,j})_{i,j\iN}$ denotes the unit matrix. Next we show that $\mu$ is a probability distribution with $\mu Q=0$. By induction, we get that
$$
1-\sum_{l=0}^k \mu_l=\prod_{l=0}^k \frac{f(l)}{1+f(l)}
$$
for any $k\iZ_+$. Since $\sum_{l=0^\infty} 1/f(l)\geq \sum_{l=0}^\infty 1/(l+1)=\infty$ it follows that $\mu$ is a probability measure on $\IZ_+$.
Moreover, it is straight-forward to verify that
$$
f(0)\mu_0=1-\mu_0=\sum_{l=1}^\infty \mu_l
$$
and that for all $k\iZ_+$
$$
f(k-1)\,\mu_{k-1}=(1+f(k))\, \mu_k,
$$
hence $\mu Q=0$.

Now we use the matrices $P^{(n)}:= I+ \frac1{n+1}\,Q $ to define an inhomogeneous Markov process. The entries of each row of $P^{(n)}$ sum up to $1$ but  (as long as $f$ is not bounded) each $P^{(n)}$ contains negative entries. Nonetheless one can use the $P^{(n)}$ as a time inhomogeneous Markov kernel as long as at the starting time $m\iN$ the starting state $l\in \IZ_+$ satisfies $l\leq m-1$.

We denote for any admissible pair $l,m$ by $(Y^{l,m}_n)_{n\geq m}$ a Markov chain starting at time $m$ in state $l$ having transition kernels $(P^{(n)})_{n\geq m}$. Due to the recursive equation we now have
$$\mu_k(n)= \IP(Y^{0,1}_n=k).$$
Next, fix $k\iZ_+$, let  $m>k$ arbitrary, and denote by $\nu$ the  restriction of $\mu$ to the set $\{m,m+1,\dots\}$. Since $\mu$ is invariant under each $P^{(n)}$ we get
$$
\mu_k=(\mu P^{(m)}\dots P^{(n)})_k= \sum_{l=0}^{m-1} \mu_l \,\IP(Y^{l,m}_n=k) + (\nu P^{(m)}\dots P^{(n)})_k.
$$

Note that in the $n$-th  step of the Markov chain, the probability to jump to state zero is $\frac1{n+1}$ for all original states in $\{1,\dots,n-1\}$ and bigger than $\frac1{n+1}$ for the original state $0$.
Thus one can couple the Markov chains $(Y^{l,m}_n)$ and $(Y^{0,1}_n)$ in such a way that
$$
\IP(Y^{l,m}_{n+1}=Y^{0,1}_{n+1}=0 \,| \, Y^{l,m}_{n}\not =Y^{0,1}_{n})=   \sfrac1{n+1},
$$
and that once the processes meet at one site they stay together.
Then
$$
\IP(Y^{l,m}_n= Y^{0,1}_n)\geq 1- \prod_{i=m}^{n-1} \frac i{i+1} \longrightarrow 1.
$$
Thus $(\nu P^{(m)}\dots P^{(n)})_k\in[0,\mu([m,\infty))]$ implies that
$$
\limsup_{n\to\infty} \Bigl|\mu_k- \IP(Y_n^{(0,1)}=k)  \sum_{l=0}^{m-1} \mu_l^*\Bigr|\leq \mu([m,\infty)).
$$
As $m\to\infty$ we thus get that
$$
\lim_{n\to\infty} \mu_k(n)= \mu_k.
$$
\medskip

In the next step we show that the sequence of the empirical indegree distributions $(X(n))_{n\iN}$ converges almost surely to $\mu$. Note that $n\,X_k(n)$ is a sum of $n$ independent Bernoulli random variables. Thus Chernoff's inequality (\cite{Che81})  implies that  for any $t>0$
$$
\IP\bigl(n\,X_k(n)\leq n\,(\IE[X_k(n)]- t)\bigr) \leq e^{-n t^2/(2\IE[X_k(n)])}= e^{-n t^2/(2\mu_k(n))}.
$$
Since
$$
\sum_{n=1}^\infty e^{-n t^2/(2\mu_k(n))}<\infty,
$$
Borel-Cantelli implies that almost surely $\liminf_{n\to\infty} X_k(n)\geq \mu_k$ for all $k\iZ_+$. This establishes almost sure convergence of $(X(n))$ to $\mu$.

\medskip

We still need to show that the conditional law of the outdegree of a new node converges almost surely in the weak topology to a Poisson distribution.
In the first step we will prove that, for $\eta\in(0,1)$, and  the  affine linear attachment rule $f(k)=\eta k+1$, one has almost sure convergence of $Y_n:=\frac 1n \sum_{m=1}^n \cZ[m,n]=\langle X(n),\mathrm{id}\rangle$ to $y:=1/(1-\eta)$.
First observe that
\begin{align*}
Y_{n+1}=\sfrac1{n+1} \bigl[ n Y_n + \sum_{m=1}^n \Delta\cZ[m,n]\bigr]= Y_n+\sfrac1{n+1} \bigl[ - Y_n + \sum_{m=1}^n \Delta\cZ[m,n]\bigr],
\end{align*}
where $\Delta\cZ[m,n]:=\cZ[m,n+1]-\cZ[m,n]$. Given the past $\cF_n$ of the network formation, each $\Delta\cZ[m,n]$ is independent Bernoulli distributed with success probability $\frac1n (\eta \cZ[m,n]+1)$. Consequently,
\begin{align*}
\IE[ Y_{n+1} |\cF_n] & = Y_n+\sfrac1{n+1} \bigl[ - Y_n + \sum_{m=1}^n \sfrac1n (\eta \cZ[m,n]+1)\bigr]\\
& =Y_n+\sfrac1{n+1} \bigl[ -(1-\eta) Y_n + 1\bigr],
\end{align*}
and
\begin{align}\label{eq1701-1}
\langle Y\rangle_{n+1} - \langle Y\rangle_n  \le  \sfrac1{(n+1)^2} \sum_{m=1}^n  \sfrac1n (\eta \cZ[m,n]+1)= \sfrac1{(n+1)^2} [\eta Y_n +1].
\end{align}
Now note that due to Theorem \ref{main} (which can be used here, as it will be proved independently of this section) there is a single node that has maximal indegree for all but finitely many times. Let $m^*$ denote the random node with this property. With Remark \ref{ReLLN} we conclude that almost surely
\begin{align}\label{eq1701-2}
\log \cZ[m^*,n]\sim \log n^{\eta}.
\end{align}
Since for sufficiently large $n$
$$
Y_n=\frac 1n \sum_{m=1}^n \cZ[m,n] \le \cZ[m^*,n],
$$
equations (\ref{eq1701-1}) and (\ref{eq1701-2}) imply  that $\langle Y\rangle_\cdot$ converges almost surely to a finite random variable.

Next, represent the increment $Y_{n+1}-Y_n$ as
\begin{align}\label{eq1701-3} Y_{n+1}-Y_n =\sfrac1{n+1} \bigl[ -(1-\eta) Y_n + 1\bigr] + \Delta M_{n+1},
\end{align}
where $\Delta M_{n+1}$ denotes a martingale difference. We shall denote by $(M_n)_{n\iN}$ the corresponding
martingale, that is $M_n=\sum_{m=2}^{n}\Delta M_m$. Since $\langle Y\rangle_\cdot$ is convergent, the martingale  $(M_n)$ converges almost surely.
Next, we represent (\ref{eq1701-3}) in terms of $\bar Y_n=Y_n-y$ as the following inhomogeneous linear difference equation of first order:
$$
\bar Y_{n+1} = \Bigl(1-\sfrac{1-\eta}{n+1}\Bigr) \bar Y_n  +\Delta M_{n+1}.
$$
The corresponding starting value is  $\bar Y_1=Y_1-y=-y$, and we can represent its solution as
$$
\bar Y_n=-y \,h_n^1+ \sum_{m=2}^n  \Delta M_m \, h_n^m
$$
for
$$
h_n^m:= \begin{cases} 0 & \text{ if }n<m\\ \prod_{l=m+1}^n \Bigl(1-\frac{1-\eta}l\Bigr)& \text{ if }n\geq m. \end{cases}
$$
Setting $\Delta h_n^m =h_n^m-h_n^{m-1}$ we conclude with an  integration by parts argument that \begin{align}\begin{split}\label{eq1801-1}
\sum_{m=2}^n  \Delta M_m \, h_n^m&=\sum_{m=2}^n  \Delta M_m \Bigl(1-\sum_{k=m+1}^n \Delta h_n^k\Bigr)=M_n-\sum_{m=2}^n \sum_{k=m+1}^n  \Delta M_m\,\Delta h_n^k \\
&=M_n- \sum_{k=3}^n \sum_{m=2}^{k-1}  \Delta M_m\,\Delta h_n^k = M_n- \sum_{k=3}^n   M_{k-1}\,\Delta h_n^k.
\end{split}
\end{align}
Note that $h_n^m$ and $\Delta h_n^m$  tend to $0$ as $n$ tends to infinity  so that $\sum_{k=m+1}^n \Delta h_n^k=1- h_n^m$ tends 
to~$1$. With $M_\infty:=\lim_{n\to\infty} M_n$ and $\eps_m=\sup_{n\ge m} |M_n-M_\infty|$ we derive for $m\leq n$
\begin{align*}
\bigl|M_n & -\sum_{k=3}^m M_{k-1} \,\Delta h_n^k -\sum_{k=m+1}^n M_{k-1} \,\Delta h_n^k\bigr|\\
& \leq \underbrace{|M_n-M_\infty|}_{\to0} + \underbrace{\sum_{k=3}^m |M_{k-1}| \,\Delta h_n^k}_{\to0} +\underbrace{| \sum_{k=m+1}^n (M_\infty-M_{k-1})\,\Delta h_n^k|}_{\leq \eps_m} + \underbrace{\bigl(1-\sum_{k=m+1}^n \Delta h_n^k\bigr) |M_\infty|}_{\to 0}.
\end{align*}
Since  $\lim_{m\to\infty} \eps_m =0$, almost surely, we thus conclude with (\ref{eq1801-1}) that $\sum_{m=2}^n  \Delta M_m \, h_n^m$ tends to $0$. Consequently, $\lim_{n\to\infty} Y_n=y$, almost surely.
Next, we show that also $\langle \mu, \mathrm{id} \rangle=y$. Recall that $\mu$ is the unique invariant distribution satisfying $\mu\,Q=0$ (see (\ref{defQ}) for the definition of $Q$). This implies that for any $k\iN$
$$
f(k-1)\,\mu_{k-1}-(f(k)+1) \,\mu_{k}=0, \ \text{ or equivalently, } \ \mu_k=f(k-1)\,\mu_{k-1}-f(k) \,\mu_{k}
$$
Thus
$$
\langle \mu,\mathrm{id}\rangle =\sum_{k=1}^\infty k\,\mu_k=\sum_{k=1}^\infty k\, \bigl[f(k-1)\,\mu_{k-1}-f(k) \,\mu_{k}\bigr].
$$
One cannot split the sum into two sums since the individual sums are not summable. However, noticing that the individual term $f(k)\mu_k k\approx k^2 \mu_k$ tends to $0$, we can rearrange the summands to obtain
$$
\langle \mu,\mathrm{id}\rangle =f(0)\,\mu_0 +\sum_{k=1}^\infty  f(k) \,\mu_{k}=\langle \mu,f\rangle = \eta \langle \mu,\mathrm{id}\rangle+1 .
$$
This implies that $\langle \mu,\mathrm{id}\rangle=y$ and that for any $m\iN$
$$
\langle X(n),\ind_{[m,\infty)}\cdot \mathrm{id}\rangle =  \langle X(n), \mathrm{id}\rangle-\langle X(n),\ind_{[0,m)}\cdot \mathrm{id}\rangle \to  \langle \mu,\ind_{[m,\infty)}\cdot\mathrm{id}\rangle,\text{ almost surely.}
$$
\medskip

Now, we switch to  general attachment rules. We denote by $f$ an arbitrary  attachment rule that is dominated by an   affine attachment rule $f^\mathrm{a}$. The corresponding degree evolutions will be denoted by $(\cZ[m,n])$ and  $(\cZ^\mathrm{a}[m,n])$, respectively. Moreover, we denote by $\mu$ and $\mu^\mathrm{a}$ the limit distributions of the empirical indegree distributions. Since by assumption $f\leq f^\mathrm{a}$, one can couple both degree evolutions such that  $\cZ[m,n]\leq \cZ^\mathrm{a}[m,n]$ for all $n\geq m\geq 0$.
Now
$$
\langle X(n), f\rangle \leq \langle X(n), \ind_{[0,m)}\cdot f\rangle + \langle X^\mathrm{a}(n), \ind_{[m,\infty)}\cdot f^\mathrm{a}\rangle
$$
so that almost surely
$$
\limsup_{n\to\infty} \langle X(n), f\rangle \leq \langle \mu, \ind_{[0,m)}\cdot f\rangle + \langle \mu^\mathrm{a}, \ind_{[m,\infty)}\cdot f^\mathrm{a}\rangle.
$$
Since $m$ can be chosen arbitrarily large we conclude that
$$
\limsup_{n\to\infty} \langle X(n), f\rangle \leq \langle \mu,  f\rangle.
$$
The converse estimate is an immediate consequence of Fatou's lemma. Hence,
$$
\lim_{n\to\infty} \IE \Bigl[\sum_{m=1}^n \Delta \cZ[m,n]\Big|\cF_n\Bigr] =\langle \mu,  f\rangle.
$$
Since, conditional on $\cF_n$,  $\sum_{m=1}^n \Delta \cZ[m,n]$ is a sum of independent Bernoulli variables with success probabilities tending uniformly to $0$, we finally get that $\cL(\sum_{m=1}^n \Delta \cZ[m,n]|\cF_n)$ converges in the weak topology to a Poisson distribution with parameter $\langle \mu,  f\rangle$.

\section{Large deviations}\label{sec5}

In this section we derive tools to analyse rare events in the random network. We provide large and
moderate deviation principles for the temporal development of the indegree of a given vertex. This
will allow us to describe the indegree evolution of the node with maximal indegree in the case of
weak preferential attachment. The large and moderate deviation principles are based on an exponential
approximation to the indegree evolution processes, which we first discuss.

\subsection{Exponentially good approximation}

In order to analyze the large deviations of the process $Z[s,\,\cdot\,]$ (or $\cZ[m,\,\cdot,\,]$) 
we use an approximating process. We first do this on the level of occupation measures. For $s\in \IT$ and $0\leq u <v$ we define
$$
T_s[u,v) = \sup\{t'-t: Z[s,t]\geq u, Z[s,t']<v, t,t'\in\IT\}
$$
to be the time the process $Z[s,\cdot]$ spends in the interval $[u,v)$. Similarly, we denote by $T_s[u]$ the time spent in $u$.
Moreover, we denote by $(T[u])_{u\in \IS}$ a family of independent random variables with each entry $T[u]$
being ${\sf Exp}(f(u))$-distributed, and denote
$$T[u,v):=\sumtwo{w\in\IS}{u\le w<v} T[w] \qquad \mbox{for all $0\leq u\leq v$.}$$
The following lemma shows that  $T[u,v)$ is a good approximation to $T_s[u,v)$ in
many cases. 

\begin{lemma}\label{appro2}
Fix $\eta_1\in(0,1)$, let $s\in\IT$ and denote by $\tau$ the entry time into $u$ of the process $Z[s,\cdot]$.
One can couple $(T_s[u])_{u\in\IS}$ and $(T[u])_{u\in\IS}$ such that, almost surely,
$$\ind_{\{\bar f(u) \Delta \tau\le \eta_1\}}    |T_s[u]- T[u]| \leq (1\vee \eta_2 \bar f(u)) \Delta \tau,$$
where $\eta_2$ is a constant only depending on $\eta_1$.
\end{lemma}

\begin{proof}
%\framebox{Proof still to be adapted}
We fix $t\in \IT$  with $\bar f(u) \Delta  t\leq \eta_1$.
Note that it suffices to find an appropriate coupling conditional on the event $\{\tau=t\}$.
Let $U$ be a uniform random variable and let $F$ and $\bar F$ denote the (conditional) distribution functions of $T[u]$ and $T_s[u]$, respectively. We couple $T[u]$ and $T_s[u]$ by setting
$T[u]=F^{-1}(U)$ and $T_s[u]=\bar F^{-1}(U)$,
where $\bar F^{-1}$ denotes the right continuous inverse of~$\bar F$.
The variables $T[u]$ and $T_s[u]$ satisfy the assertion of the lemma if and only if
\begin{align}\label{eq0802-1}
F\big(v- (1\vee \eta_2\bar f(u)) \Delta t\big)\leq \bar F(v)\leq F\big(v + (1\vee \eta_2\bar f(u)) \Delta t \big)\qquad
\mbox{ for all } v\geq 0.
\end{align}
We compute
$$1-\bar F(v)  = \prodtwo{t\le w, w+\Delta w\le t+v}{w\in\IT} \big(1-\bar f(u) \Delta w\big) =
\exp \sumtwo{t\le w, w+\Delta w\le t+v}{u\in\IT} \log \big(1-\bar
f(u) \Delta w\big)$$ Next observe that, from a Taylor expansion, for
a suitably large $\eta_2>0$, we have $-\bar f(u) \Delta w- \eta_2\,
\bar f(u)^2 [\Delta w]^2 \leq \log \big(1-\bar f(u) \Delta w\big)
\leq -\bar f(u) \Delta w,$ so that
\begin{align*}
1-\bar F(v)  & \leq \exp \Bigl(-\bar f(u)\sumtwo{t\le w, w+\Delta
w\le t+v}{w\in\IT}  \Delta w\Bigr)  \leq \exp\big(-\bar f(u)
(v-\Delta t)\big) = 1- F(v-\Delta t).
\end{align*}
This proves the left inequality in (\ref{eq0802-1}). It remains to prove the right inequality.
Note that
\begin{align*}
1-\bar F(v) &\geq \exp \Bigl(-\sumtwo{t\le w, w+\Delta w\le
t+v}{w\in\IT}  (\bar f(u) \Delta w+ \eta_2\,\bar f(u)^2[\Delta w]^2)
\Bigr)
\end{align*}
and
$$\sumtwo{t\le w, w+\Delta w\le t+v}{w\in\IT}  [\Delta w]^2\leq
\sum_{m=[\Delta t]^{-1}}^\infty \frac1{m^2}\leq \frac1{[\Delta
t]^{-1}-1}\leq \Delta t.$$ Consequently, $1-\bar F(v)   \geq \exp
\{-\bar f(u)(v+\eta_2\bar f(u) \Delta t)\}
 = 1- F(v+\eta_2\bar f(u) \Delta t)$.
\end{proof}
\smallskip

As a direct consequence of this lemma we obtain an exponential approximation.

\begin{lemma}\label{le1025-1}
Suppose that, for some $\eta<1$ we have  $f(j)\leq \eta(j+1)$ for all~$j\in\IZ_+$.  
If
$$\sum_{j=0}^\infty \frac{f(j)^2}{(j+1)^2} < \infty,$$
then for each $s\in\IT$ one can couple $T_s$ with $T$ such that, for all $\lam\geq 0$,
$$%\label{eq0808-1}
\IP\big(\sup_{u\in\IS}  |T_s[0,u)-T[0,u)|\geq \lam+\sqrt {2K}\big)
\leq 4 \,e^{-\frac{\lam^2}{2K}},$$
where $K>0$ is a finite constant only depending on $f$.
\end{lemma}

\begin{proof}
Fix $s\in\IT$ and denote by $\tau_u$ the first entry time of $Z[s,\cdot]$ into the state $u\in\IS$.
We couple the random variables $T[u]$ and $T_s[u]$ as in the previous lemma and let, for $v\in\IS$,
$$
M_v=\sum_{\heap{u\in \IS}{u< v}}  \big(T_s[u]-T[u]\big)=T_s[0,v)-T[0,v).
$$
Then $(M_v)_{v\in \IS}$ is a martingale. Moreover, for each $v=\Phi(j)\in\IS$ one has
$\tau_v\geq \Psi(j+1)$ so that $\Delta \tau_v\leq 1/(j+1)$. Consequently, using the 
assumption of the lemma one gets that
$$\Delta \tau_v \,\bar f(v)\leq  \frac1{j+1} f(j)=:c_v\leq \eta<1.$$
Thus by Lemma \ref{appro2} there exists a constant $\eta'<\infty$ depending only on $f(0)$ and $\eta$ such that
the increments of the martingale  $(M_v)$  are bounded by
$$
|T_s[v]-T[v]| \leq \eta'\, c_v.
$$
By assumption we have $K:=\sum_{v\in \IS} c_v^2<\infty$ an we conclude with Lemma \ref{conc2}  that for $\lam\geq 0$,
\begin{align*}%\label{eq0808-1}
\IP\big(\sup_{u\in\IS}  |T_s[0,u)-T[0,u)|\geq \lam+\sqrt {2K}\big)
\leq 4 \,e^{-\frac{\lam^2}{2K}}.
\end{align*}
\end{proof}

We define $(Z_t)_{t\geq 0}$ to be the $\IS$-valued process given by
\begin{equation}\label{Zdef}
Z_t:= \max \{ v\in\IS: T[0,v)\leq t\},
\end{equation}
and start by observing its connection to the indegree evolution.

\begin{corollary}
In distribution on the Skorokhod space, we have
$$
\lim_{s\uparrow \infty} (Z[s,s+t])_{t\geq0} = (Z_t)_{t\geq 0}.
$$
%as $s$ tends to infinity.
\end{corollary}

\begin{proof}
Recall that Lemma \ref{appro2} provides a coupling between
$(T_s[u])_{u\in\IS}$ and $(T[u])_{u\in\IS}$ for any fixed $s\in\IT$.
We may assume that the coupled random variables
$(\bar T_s[u])_{u\in\IS}$ and $(\bar T[u])_{u\in\IS}$ are defined
on the same probability space for all $s\in\IT$ (though this
does not respect the joint distributions of $(T_s[u])_{u\in\IS}$
for different values of $s\in\IT$). Denote by $(\bar
Z[s,\,\cdot\,])_{ s\in\IT}$ and $(\bar Z_t)$ the corresponding processes such that
$\bar T_s[0,u)+s=s+\sum_{v<u} \bar T_s[v]$ and $\bar T[0,u)=\sum_{v<u} \bar T[v]$ are
the entry times of $(\bar Z[s,s+t])$ and $(\bar Z_t)$ into the state
$u\in\IS$. By Lemma \ref{appro2} one has that, almost surely,
$\lim_{s\uparrow\infty} \bar T_s[0,u) = \bar T[0,u)$
and therefore one obtains almost sure convergence of
$(\bar Z[s,s+t])_{t\ge 0}$ to $(\bar Z_t)_{t\ge 0}$
in the Skorokhod topology.
\end{proof}

%\begin{lemma}
%For $\lam>0$ let $T^\lam[u]$ be independent ${\sf Exp}(\lam f(u))$-distributed random variables,
%and denote by~$F$ their distribution function. Suppose that for all~$j\iZ_+$
%$$1-\frac{F(j)}{j+1}\geq \frac1\lam.$$
%Then one can couple $T_s$ and $T^\lam$ such that for any $u\in\IS$
%$$T_s[u]\geq T^\lam[u].$$
%\end{lemma}

%\begin{proof}We start as in the proof of Lemma~\ref{appro2}.
%We fix $u=\Phi(j)\in\IS$ and $t\in \IT$ with $\Delta t\leq 1/(j+1)$. Note that it suffices to find an appropriate coupling conditional on the event $\{\tau=t\}$.
%Let $U$ be a uniform random variable and let $F$ and $\bar F$ denote the (conditional) distribution functions of $T^\lam[u]$ and $T_s[u]$, respectively. We couple $T^\lam[u]$ and $T_s[u]$ by setting
%$$
%T[u]=F^{-1}(U)\quad\text{ and }\quad T_s[u]=\bar F^{-1}(U),
%$$
%where $\bar F^{-1}$ denotes the generalized inverse of $\bar F$. Then the assertion is a consequence of the following computation:
%\begin{align*}
%1-\bar F(v) & = \prodtwo{t\leq w, w+\Delta w\leq t+v}{w\in\IT} (1-\bar f(u) \Delta w) = \exp \sumtwo{t\leq w, w+\Delta w\leq t+v}{u\in\IT} \log (1-\bar f(u) \Delta w)\\
%&\geq \exp \Bigl(-\sumtwo{t\leq w, w+\Delta w\leq t+v}{u\in\IT} \lam \bar f(u) \Delta w\Bigr)\geq \exp(-\lam \bar f(u) v)=1-F(v).
%\end{align*}
%\end{proof}

\begin{prop}\label{expoapp}
Uniformly in~$s$, the processes 
\begin{itemize}
\item $(\frac1\kappa Z_{\kappa t}\colon t\geq 0)_{\kappa>0}$ and $(\frac1\kappa Z[s,s+\kappa t]\colon t\geq 0)_{\kappa>0}$;
\item $(\frac1{a_\kappa} (Z_{\kappa t}-\kappa t) \colon t\geq 0)_{\kappa>0}$ and 
$(\frac1{a_\kappa} (Z[s,s+\kappa t]-\kappa t) \colon t\geq 0)_{\kappa>0}$,
\end{itemize}
are exponentially equivalent on the scale of the large, respectively, moderate
deviation principles.
\end{prop}

\begin{proof} We only present the proof for the first large deviation principle of Theorem~\ref{LDP} since all other statements can be inferred analogously.

We let $U_\delta(x)$ denote the open ball around $x\in \cI[0,\infty)$ with radius $\delta>0$  in an arbitrarily fixed metric $d$ generating the topology of uniform convergence on compacts, and, for fixed $\eta>0$, we cover the compact set $K=\{x\in\cI[0,\infty)\colon I(x)\leq \eta\}$ with finitely many balls $(U_\delta(x))_{x\in\II}$, where $\II\subset K$. 
Since every $x\in\II$ is continuous, we  can find $\eps>0$ such that for every $x\in\II$ and increasing and right continuous $\tau:[0,\infty)\to[0,\infty)$ with $|\tau(t)-t|\leq \eps$,
$$
y\in U_\delta(x) \ \Rightarrow \ y_{\tau(\cdot)}\in U_{2\delta}(x).
$$
For fixed $s\in\IT$ we couple the occupation times $(T_s[0,u))_{u\in\IS}$ and $(T[0,u))_{u\in\IS}$ as in Lemma~\ref{le1025-1},
and hence implicitly the evolutions $(Z[s,t])_{t\ge s}$ and $(Z_t)_{t\ge 0}$.  Next, note that $Z[s,s+\,\cdot\,]$  can be transformed into 
$Z_\cdot$ by applying a time change $\tau$ with $|\tau(t)-t|\leq \sup_{u\in\IS} |T_s[0,u)-T[0,u)|$.
Consequently,
\begin{align*}
\IP\big(d\big(\sfrac1\kappa Z[s,s+\kappa\,\cdot\,], \sfrac1\kappa Z_{\kappa\cdot}\big)\geq 3\delta\big) 
\leq \IP\Bigl(\sfrac1\kappa Z_{\kappa\cdot}\not\in \bigcup_{x\in\II} U_{\delta}(x)\Bigr) +\IP\Big( \sup_{u\in\IS} |\bar T_s[0,u)-\bar T[0,u)|\geq 
\kappa\eps\Big),
\end{align*}
and an application of Lemma~\ref{le1025-1} gives a uniform upper bound in $s$, namely
$$
\limsup_{\kappa\to \infty} \sup_{s\in\IS} \frac 1{\kappa^{\frac1{1-\alpha}}\bar\ell(\kappa)} \log \IP\bigl(d\big(\sfrac1\kappa Z[s,s+\kappa\,\cdot\,], \sfrac1\kappa Z_{\kappa\cdot}\big)\geq 3\delta\bigr)\leq -\eta.
$$
Since $\eta$ and $\delta>0$ were arbitrary this proves the first statement.
\end{proof}

\subsection{The large deviation principles}\label{sec42}

By the exponential equivalence, Proposition~\ref{expoapp}, and \cite[Theorem~4.2.13]{DemZei98} it suffices to prove the large and moderate
deviation principles in the framework of the exponentially equivalent processes~\eqref{Zdef} constructed 
in the previous section.
\smallskip

The first step in the proof of the first part of Theorem~\ref{LDP},
is to show a large deviation principle for the occupation times
of the underlying process.  Throughout this section we denote
$$a_\kappa:=\kappa ^{1/(1-\alpha)} \bar \ell(\kappa).$$
We define the function $\xi\colon\IR\to(-\infty,\infty]$ by
$$\xi(u)=\begin{cases}
\log \frac1{1-u} & \text{ if } u<1,\\ \infty &\text{ otherwise}.
\end{cases}$$
Its Legendre-Fenchel transform is easily seen to be
$$
\xi^*(t)=\begin{cases} t-1-\log t & \text{ if }t>0,\\ \infty &
\text{ otherwise.}
\end{cases}
$$

\begin{lemma}\label{occuLDP}
For fixed $0\leq u<v$ the family $(\frac1{\kappa}T{[\kappa u,\kappa
v)})_{\kappa>0}$ satisfies a large deviation principle with speed
$(a_\kappa)$ and rate
function  $\Lam^*_{u,v}(t)=\sup_{\zeta\iR} [
t\zeta-\Lam_{u,v}(\zeta)]$, where
$$
\Lam_{u,v}(\zeta)=\int_{u}^v s^{\frac{\alpha}{1-\alpha}}\xi(\zeta
s^{-\alpha/(1-\alpha)}) \,ds.
$$
\end{lemma}

\begin{proof}
For fixed $u<v$ denote by $\II_\kappa=\II^{[u,v)}_\kappa=\{j\in\IZ_+
\colon \Phi(j)\in [\kappa u,\kappa v)\}$. We get, using $(S_j)$ for
the underlying sequence of ${\sf Exp}(f(j))$-distributed independent
random variables,
\begin{align*}
\Lam_\kappa(\th):=\ &\log \IE e^{\th T{[\kappa u,\kappa v)}/\kappa}\\
=\ &\sum_{j\in \II_\kappa} \log \IE e^{\frac{\th}{\kappa} S_j}= \sum_{j\in \II_\kappa} \log \frac{1}{1-\frac{\th}{\kappa f(j)}}= \sum_{t\in \Phi(\II_\kappa)} \xi\Big(\frac{\th}{\kappa f(\Phi^{-1}(t))}\Big)\\
=\ & \int_{\bar\II_\kappa} f(\Phi^{-1}(t))\, \xi\Big(\frac{\th}{\kappa
f(\Phi^{-1}(t))}\Big) \,dt,
\end{align*}
where $\bar\II_\kappa=\bar \II_\kappa^{[u,v)}=
\bigcup_{j\in\II_\kappa} [\Phi(j),\Phi(j+1))$. Now choose $\th$ in
dependence on $\kappa$ as $\th _\kappa=\zeta \kappa ^{1/(1-\alpha)}
\bar \ell(\kappa)$ with $\zeta<u^{\alpha/(1-\alpha)}$. Then
\begin{align*}
\int_{\bar\II_\kappa} \bar f(t)\xi\Big(\frac{\th_\kappa}{\kappa \bar f(t) }\Big) \,dt&=\kappa \int_{\bar\II_\kappa/\kappa} \bar f(\kappa s) \xi\Big(\frac{\th_\kappa}{\kappa \bar f(\kappa s)}\Big) \,ds\\
&=\kappa^{1/(1-\alpha)}\int_{\bar\II_\kappa/\kappa}
s^{\frac{\alpha}{1-\alpha}} \bar\ell(\kappa s) \xi \Big(\frac{\zeta
\bar \ell(\kappa)}{s^{\frac{\alpha}{1-\alpha}}  \bar \ell(\kappa
s)}\Big) \,ds.
\end{align*}
Note that $\inf \bar\II_\kappa/\kappa$ and $\sup \bar\II_\kappa/\kappa$
approach the values $u$ and $v$, respectively. Hence, we conclude
with the dominated convergence theorem that  one has
$$
\Lam_\kappa(\th_\kappa)\sim \kappa^{1/(1-\alpha)} \bar \ell(\kappa )
\underbrace{\int_{u}^v s^{\frac{\alpha}{1-\alpha}}\xi
(\mbox{$\frac{\zeta }{s^{\frac{\alpha}{1-\alpha}}}$})
\,ds}_{=\Lam_{u,v}(\zeta)}
$$
as $\kappa$ tends to infinity. Now the G\"artner-Ellis theorem
implies the large deviation principle for the family $(T{[\kappa
u,\kappa v)})_{\kappa>0}$ for $0<u<v$. It remains to prove the large
deviation principle for $u=0$. Note that
$$
\IE T{[0,\kappa v)}=\IE \sum_{j\in \II_\kappa} S_j
=\int_{\bar\II_\kappa} f(\Phi^{-1}(t)) \frac1{f(\Phi^{-1}(t))}
\,dt\sim \kappa v
$$
and
$$
\var (T{[0,\kappa v)})=\sum_{j\in \II_\kappa} \var(
S_j)=\int_{\bar\II_\kappa} f(\Phi^{-1}(t)) \frac1{f(\Phi^{-1}(t))^2}
\,dt\lesssim \frac1{f(0)} \kappa v.
$$
Consequently, $\frac{T{[0,\kappa \eps)}}\kappa$ converges in
probability to $\eps$. Thus for $t<v$
$$
\IP\Bigl(\sfrac1\kappa T{[0,\kappa v)} \leq t\Bigr) \geq
\underbrace{\IP\Bigl(\sfrac1\kappa T{[0,\kappa\eps)}\leq
(1+\eps)\eps\Bigr)}_{\to1} \IP\Bigl(\sfrac1\kappa T{[\kappa
\eps,\kappa v)}\leq t-(1+\eps)\eps\Bigr)
$$
and for sufficiently small $\eps>0$
$$
\liminf_{\kappa\to \infty} \frac1{a_\kappa}
\log\IP\Bigl(\sfrac1\kappa T{[0,\kappa v)} \leq t\Bigr)\geq
-\Lam_{\eps,v}^*(t-(1+\eps)\eps),
$$
while the upper bound is obvious.
\end{proof}

The next  lemma is necessary for the analysis of the rate function
in Lemma~\ref{occuLDP}. It involves the function $\psi$ defined as $\psi(t)=1-t+t\log t$ for $t\geq 0$.

\begin{lemma}\label{simplJ}
For fixed  $0<x_0<x_1$ there exists an increasing function $\eta\colon\IR_+\to \IR_+$ with $\lim_{\delta\dto0}
\eta_\delta=0$ such that for any $u,v\in[x_0,x_1]$ with $\delta :=v-u> 0$ and all $w\in[u,v], t> 0$ one has
$$\Bigl|\Lam^*_{u,v}(t)-  w^{\frac{\alpha}{1-\alpha}} \,t\, \psi\bigl(\sfrac\delta t\bigr)\Bigr|  
\leq \eta_\delta \Bigl(\delta +t \,\psi\bigl(\sfrac \delta t\bigr)\Bigr).$$
\end{lemma}
\medskip

We now extend the definition of $\Lam^*$ continuously by setting, for any~$u\geq 0$ and $t \geq 0$,
$$\Lam^*_{u,u}(t)=  u^{\frac{\alpha}{1-\alpha}} \,t .$$
For the proof of Lemma~\ref{simplJ} we use the following fact, which can be verified easily.

\begin{lemma}\label{le0821-1}
For any $\zeta>0$ and $t>0$, we have
$|\xi^*(\zeta t)-\xi^*(t)| \leq 2|\zeta-1|+|\log \zeta| +2|\zeta-1|
\xi^*(t).$
\end{lemma}

%We estimate, for $t\in(0,\infty)$,
%$$|\xi^*(\zeta t)-\xi^*(t)|=|(\zeta-1)t-\log \zeta| \leq
%2|\zeta-1|+|\log \zeta| +|\zeta-1|(t-2)_+.$$
%Since $(\xi^*)'(2)=1/2$  and $\xi^*$ is convex and non-negative, we
%get that $(t-2)_+\leq 2 \xi^*(t)$ which immediately implies the
%statement.

\begin{proof}[ of Lemma~\ref{simplJ}]
First observe that
$$\gamma_\delta:=\suptwo{x_0<u<v<x_1}
{v-u\leq \delta} (v/u)^{\frac{\alpha}{1-\alpha}}$$ 
tends to $1$ as $\delta$ tends to zero. 
By Lemma \ref{le0821-1}, there exists a function $(\bar \eta_\delta)_{\delta>0}$
with $\lim_{\delta\dto0}\bar \eta_\delta =0$ such that for all
$\zeta\in [1/\gamma_\delta,\gamma_\delta]$ and $t>0$
$$|\xi^*(\zeta t)-\xi^*(t)|\leq \bar\eta_\delta (1+\xi^*(t)) .$$
Consequently, one has for any $\delta>0$, $x_0<w,\bar w<x_1$ with
$|w-\bar w|\leq \delta$ and $\zeta\in
[1/\gamma_\delta,\gamma_\delta]$ that
\begin{align*}
|\bar w^{\frac{\alpha}{1-\alpha}} \xi^*(\zeta t)-w^{\frac{\alpha}{1-\alpha}} \xi^*(t)|&\leq \bar w^{\frac{\alpha}{1-\alpha}} |\xi^*(\zeta t)-\xi^*(t)| + \xi^*(t) |\bar w^{\frac{\alpha}{1-\alpha}}-w^{\frac{\alpha}{1-\alpha}}|\\
&\leq c \bar \eta_\delta (1+\xi^*(t))+ c \delta \xi^*(t),
\end{align*}
where $c<\infty$ is a constant only depending on $x_0,x_1$ and
$\alpha$. Thus for an appropriate function
$(\eta_\delta)_{\delta>0}$ with $\lim_{\delta\dto0} \eta_\delta=0$
one gets
\begin{align}\label{eq0815-1}
|\bar w^{\frac{\alpha}{1-\alpha}} \xi^*(\zeta
t)-w^{\frac{\alpha}{1-\alpha}} \xi^*(t)| \leq \eta_\delta
(1+\xi^*(t)).
\end{align}
Fix $x_0<u<v<x_1$ and set $\delta:=v-u$. We estimate,  for $\th\geq0$,
$$ \delta u^{\frac{\alpha}{1-\alpha}} \xi(\th v^{-\alpha/(1-\alpha)}) \leq
\Lam_{u,v}(\th) \leq \delta v^{\frac{\alpha}{1-\alpha}} \xi(\th u^{-\alpha/(1-\alpha)}),$$
and the reversed inequalities for $\th\leq 0$. Consequently,
\begin{align*}
\Lam^*_{u,v}(\delta t)&=\sup_{\th}[\th t-\Lam_{u,v}(\th)]\\
&\leq \delta \sup_\th [\th t-u^{\frac{\alpha}{1-\alpha}} \xi(\th
v^{-\alpha/(1-\alpha)})]
 \vee \delta \sup_\th [\th t-v^{\frac{\alpha}{1-\alpha}} \xi(\th u^{-\alpha/(1-\alpha)})]\\
&= \delta u^{\frac{\alpha}{1-\alpha}}
\xi^*((v/u)^{\frac{\alpha}{1-\alpha}} t)\vee \delta
v^{\frac{\alpha}{1-\alpha}} \xi^*((u/v)^{\frac{\alpha}{1-\alpha}} t).
\end{align*}
Since $(v/u)^{\alpha/(1-\alpha)}$ and $(u/v)^{\alpha/(1-\alpha)}$
lie in $[1/\gamma_\delta,\gamma_\delta]$ we conclude with
(\ref{eq0815-1}) that for $w\in[u,v)$
$$ \Lam^*_{u,v}(\delta t)\leq w^{\frac{\alpha}{1-\alpha}} \xi^*(t)
\delta + \eta_\delta (1+ \xi^*(t))\delta.$$
To prove the converse inequality, observe
\begin{align*}
\Lam^*_{u,v}(t) &\geq \Big( \delta \sup_{\th\le0} [\th
t-u^{\frac{\alpha}{1-\alpha}} \xi(\th v^{-\alpha/(1-\alpha)})] \Big) \vee
\Big( \delta \sup_{\th\ge0} [\th t-v^{\frac{\alpha}{1-\alpha}} \xi(\th
u^{-\alpha/(1-\alpha)})] \Big).
\end{align*}
Now note that the first partial Legendre transform can be replaced
by the full one if $t\leq (u/v)^{\alpha(1-\alpha)}$. Analogously,
the second partial Legendre transform can be replaced if $t\geq
(v/u)^{\alpha(1-\alpha)}$. Thus we can proceed as above if $t\not\in
(1/\gamma_\delta,\gamma_\delta)$ and conclude that
$$\Lam^*_{u,v}(t)\geq w^{\frac{\alpha}{1-\alpha}} \xi^*(t) \delta -
\eta_\delta (1+ \xi^*(t))\delta.$$
The latter estimate remains valid on $(1/\gamma_\delta,\gamma_\delta)$ if
$x_1^{\alpha/(1-\alpha)} (\xi^*(1/\gamma_\delta)\vee
\xi^*(\gamma_\delta)) \leq \eta_\delta.$
Since $\gamma_\delta$ tends to $1$ and $\xi^*(1)=0$ one can make
$\eta_\delta$ a bit larger to ensure that the latter estimate is
valid and $\lim_{\delta\dto0} \eta_\delta =0$. This establishes the
statement.
\end{proof}

As the next step in the proof of Theorem~\ref{LDP} 
we formulate a finite-dimensional large deviation principle, which can be derived from
Lemma~\ref{occuLDP}.

\begin{lemma}
Fix $0=t_0<t_1<\cdots<t_p$. Then the vector
$$\Big( \sfrac1{\kappa} Z_{\kappa t_j}  \, :\, j\in\{1,\ldots,p\}\Big)$$
satisfies a large deviation principle in $\{0\leq a_1 \leq \cdots \leq a_p\} \subset \IR^p$ 
with speed $a_\kappa$ and rate function
$$J(a_1,\ldots,a_p)=\sum_{j=1}^p \Lambda^*_{a_{j-1},a_j}(t_{j}-t_{j-1}), \qquad \mbox{ with }a_0:=0\, .$$
\end{lemma}

\begin{proof}
First fix $0=a_0<a_1< \cdots <a_p$.
Observe that, %whenever $a_j < b_j$ with $b_0:=0$,  we have
whenever $s_{j-1}<s_j$ with $s_0=0$,
$$\begin{aligned}
\IP\big( \sfrac1\kappa  \, Z_{\kappa t_j} & \geq a_j > \sfrac1\kappa  \, Z_{\kappa s_j} 
\mbox{ for } j\in\{1,\ldots,p\} \big) \\
& \geq \IP\big( s_j-s_{j-1} < \sfrac1\kappa T{[a_{j-1}\kappa, a_j\kappa)}  \leq t_j-t_{j-1}
\mbox{ for } j\in\{1,\ldots,p\} \big).
\end{aligned}
$$
Moreover, supposing that $0<t_j-t_{j-1}-(s_j-s_{j-1})\leq \delta$ for a $\delta>0$, we obtain
\begin{align*}
\IP\big( a_j\leq \sfrac1\kappa &Z_{\kappa t_j} <a_j+\eps \mbox{ for } j\in\{1,\ldots,p\} \big)\\
&\geq \IP(\sfrac1\kappa  \, Z_{\kappa t_j}  \geq a_j > \sfrac1\kappa  \, Z_{\kappa s_j} 
\text{ and } T[a_j\kappa,(a_j+\eps)\kappa) \geq \delta \mbox{ for } j\in\{1,\ldots,p\}\big)
\end{align*}
%%we have
%$$\IP\big( \sfrac1\kappa T{[a_j\kappa, b_j\kappa)}  \geq \delta\big) \leq e^{-A a_\kappa}.$$

By Lemma~\ref{occuLDP}, given $\eps>0$ and $A>0$, we find $\delta>0$ such that, for $\kappa$ large,
$$\IP\big( \sfrac1\kappa T{[a_{j}\kappa, (a_j+\eps)\kappa)}  < \delta\big) \leq e^{-A a_\kappa}.$$
Hence, for sufficiently small $\delta$ we get with the above estimates that
$$\begin{aligned} 
\liminf_{\kappa\to\infty} \frac1{a_\kappa}&\log \IP\big( a_j+\eps > \sfrac1\kappa Z_{\kappa t_j} \geq a_j \mbox{ for } j\in\{1,\ldots,p\}\big)\\
&\geq \liminf_{\kappa\to\infty} \frac1{a_\kappa}\log \IP\big( s_j-s_{j-1} < \sfrac1\kappa T{[a_{j-1}\kappa, a_j\kappa)}  \leq t_j-t_{j-1}
\mbox{ for } j\in\{1,\ldots,p\} \big)\\
&\geq -\sum _{j=1}^p \Lam^*_{a_{j-1}, a_j}(t_j-t_{j-1}).
\end{aligned}$$

Next, we prove the upper bound. Fix $0=a_0\leq \dots \leq a_p$ and $0=b_0\leq \dots \leq b_p$ with $a_j<b_j$, 
and observe that by the strong Markov property of $(Z_t)$,
\begin{align*}
\IP\big( b_j > & \sfrac1\kappa  \, Z_{\kappa t_j} \geq a_j \mbox{ for } j\in\{1,\ldots,p\} \big)\\
&=  \prod_{j=1}^p \IP\big( b_j >  \sfrac1\kappa  \, Z_{\kappa t_j} \geq a_j\,\big|\, b_i >  \sfrac1\kappa  \, Z_{\kappa t_i} \geq a_i \mbox{ for } i\in\{1,\ldots,j-1\} \big)\\
&\leq \prod_{j=1}^p \IP\big( \sfrac1\kappa T[b_{j-1} \kappa,a_j \kappa) < t_j-t_{j-1}\leq \sfrac1\kappa T[a_{j-1}\kappa,b_j \kappa) \big).
\end{align*} 
Consequently,
$$
\limsup_{\kappa\uparrow\infty}
\frac1{a_\kappa} \log\IP\big(b_j >  \sfrac1\kappa  \, Z_{\kappa t_j} \geq a_j \mbox{ for } j\in\{1,\ldots,p\} \big) \leq -\sum_{j=1}^p r_j,
$$
where 
$$
r_j=\begin{cases}
\Lam^*_{b_{j-1},a_j} (t_j-t_{j-1}) \quad & \text{ if }a_j-b_{j-1}\geq t_j-t_{j-1},\\
\Lam^*_{a_{j-1},b_j} (t_j-t_{j-1}) & \text{ if }b_j-a_{j-1}\leq t_j-t_{j-1},\\
0,& \text{ otherwise.}
\end{cases}
$$
Using the continuity of $(u,v)\mapsto \Lam^*_{u,v}(t)$ for fixed $t$, it is easy to verify continuity of each $r_j$ of the parameters $a_{j-1}$, $a_j$, $b_{j-1}$, and $b_j$. Suppose now that  $(a_j)$ and $(b_j)$ are taken from a predefined compact subset of $\IR^d$. Then we have
$$
\sum_{j=1}^p \big | r_j- \Lam^*_{a_{j-1},a_j}(t_j-t_{j-1})\big |\leq \vth \big(\max\{b_j-a_j:j=1,\dots,p\}\big),
$$ 
for an appropriate function $\vth$ with $\lim_{\delta\dto0}\vth(\delta)=0$. Now the upper bound 
follows with an obvious exponential 
tightness argument.
\end{proof}

We can now prove a large deviation principle in a weaker topology, by taking a projective limit 
and simplifying the resulting rate function with the help of Lemma~\ref{simplJ}.
\smallskip

\begin{lemma}\label{weakldp}
On the space of increasing functions with the topology of pointwise convergence
the family of functions
$$\Big( \sfrac1{\kappa} Z_{\kappa t} \colon t\geq 0 \Big)_{\kappa>0}$$
satisfies a large deviation principle with speed~$(a_\kappa)$ and rate function~$J$.
\end{lemma}

\begin{proof}
Observe that the space of increasing functions equipped with the topology of
pointwise convergence can be interpreted as projective limit of the spaces 
$\{0\leq a_1 \leq \cdots \leq a_p\}$ with the canonical projections given by
$\pi(x)=(x(t_1),\ldots,x(t_p))$ for $0<t_1<\ldots<t_p$. By the
Dawson-G\"artner theorem, we obtain a large deviation principle with
good rate function
$$\tilde J(x) = \sup_{0<t_1<\ldots<t_p} 
\sum_{j=1}^p \Lambda^*_{x(t_{j-1}),x(t_j)}(t_{j}-t_{j-1}).$$
Note that the value of the variational expression is nondecreasing,
if additional points are added to the partition. It is not hard to see
that $\tilde J(x)=\infty$, if $x$ fails to be 
% continuous. If $x$ is continuous, but not
absolutely continuous.
 
Indeed, there exists $\delta>0$ and, for every $n\in\IN$, a partition
$\delta\leq s_1^n<t_1^n \leq\cdots\leq s_n^n<t^n_n\leq \frac1\delta$ such that
$\sum_{j=1}^n t^n_j-s^n_j \to 0$ but $\sum_{j=1}^n x(t^n_j)-x(s^n_j) \geq \delta$.
Then, for any $\lam>0$,
$$\begin{aligned}
\tilde J(x) & = \sup_{\heap{0<t_1<\ldots<t_p}{\lambda_1, \ldots, \lambda_p\in\IR}} 
\sum_{j=1}^p \lambda_j\big(t_j-t_{j-1}\big)  -  \Lambda_{x(t_{j-1}),x(t_j)}(\lambda_{j})\\
& \geq 
\sum_{j=1}^n \Big[-\lambda \big(t^n_j- s^n_j \big) + \int_{x(s^n_j)}^{x(t^n_{j})} u^{\frac\alpha{1-\alpha}}
\log\big( 1+\lambda u^{\frac{-\alpha}{1-\alpha}}\big) \, du \Big] \\
&\geq -\lam \sum_{j=1}^n \big(t^n_j- s^n_j \big) +\delta^{\frac1{1-\alpha}} \,\log\big( 1+\lambda \delta^{\frac{\alpha}{1-\alpha}}\big)
\longrightarrow \delta^{\frac1{1-\alpha}} \log\big( 1+\lambda \delta^{\frac{\alpha}{1-\alpha}}\big),
\end{aligned}$$
which can be made arbitrarily large by choice of~$\lam$.

From now on suppose that $x$ is absolutely continuous.
The remaining proof is based on the equation 
\begin{align}\label{eq:approx_J}
\tilde J(x) &  = \sup_{0<t_1<\ldots<t_p}  \sum_{j=1}^p \big(t_{j}-t_{j-1}\big)\, x(t_{j})^{\frac{\alpha}{1-\alpha}} \,
\psi\Big(\frac{x(t_{j})-x(t_{j-1})}{t_j-t_{j-1}}\Big).
\end{align}
Before we prove its validity we apply (\ref{eq:approx_J}) to derive the assertions of the lemma.
For  the \emph{lower bound} we choose a
scheme $0<t^n_1<\cdots <t_n^p$, with $p$ depending on~$n$, such that $t_n^p\to\infty$
and the mesh goes to zero. Define, for $t^n_{j-1}\leq t< t^n_j$, 
$$x^n_j(t)= \frac1{t^n_j-t^n_{j-1}} \, \int_{t^n_{j-1}}^{t^n_j} {\dot x}_s \, ds = 
\frac{x(t^n_{j})-x(t^n_{j-1})}{t^n_j-t^n_{j-1}}.$$
Note that, by Lebesgue's theorem, $x^n_j(t) \to {\dot x}_t$ almost everywhere. Hence
$$\begin{aligned}
\tilde J(x) & \geq \liminf_{n\to\infty} \int_0^{t^n_p} x_t^{\frac{\alpha}{1-\alpha}} \,
\psi\big(x^n_j(t)\big) \, dt 
\geq \int_0^\infty x_t^{\frac{\alpha}{1-\alpha}} \, \liminf_{n\to\infty}
\psi\big(x^n_j(t)\big) \, dt = J(x). 
\end{aligned}$$

For the \emph{upper bound} we use the convexity of $\psi$ to obtain
$$\begin{aligned}
\psi\Big(\frac{x(t_{j})-x(t_{j-1})}{t_j-t_{j-1}}\Big) 
= \psi\Big(\frac{1}{t_{j}-t_{j-1}} \int_{t_{j-1}}^{t_j} {\dot x}_t \, dt\Big) 
\leq \frac{1}{t_{j}-t_{j-1}} \int_{t_{j-1}}^{t_j} \psi\big({\dot x}_t\big) \, dt.
\end{aligned}$$
Hence
$$\begin{aligned}
\tilde J(x) & \leq \sup_{0<t_1<\ldots<t_p} \sum_{j=1}^p  x(t_{j})^{\frac{\alpha}{1-\alpha}} \,
\int_{t_{j-1}}^{t_j} \psi\big({\dot x}_t\big) \, dt= J(x)\, ,
\end{aligned}$$
as required to complete the proof.

It remains to prove (\ref{eq:approx_J}). We fix $t'$ and $t''$   with $t'<t''$ and $x(t')>0$, and partitions $t'=t_0^n<\dots<t_n^n=t''$ with $\delta_n:=\sup_{j} x(t_j^n)-x(t_{j-1}^n)$ converging to $0$. 
Assume $n$ is sufficiently large such that $\eta_{\delta_n}\leq \frac12 (t')^{\frac\alpha{1-\alpha}}$, 
with $\eta$ as in Lemma \ref{simplJ}. Then,
\begin{equation}\begin{aligned}\label{eq_bdd}
\sum_{j=1}^n & \Lam^*_{x(t^n_{j-1}),x(t^n_{j})}(t_j^n-t_{j-1}^n)\\
& \geq  \frac12 (t')^{\frac\alpha{1-\alpha}} \Bigl[\underbrace{\sum_{j=1}^n   (t_j^n-t_{j-1}^n) \,\psi\bigl(\frac{x(t_j^n)-x(t_{j-1}^n)}{ t_j^n-t_{j-1}^n}\bigr)}_{(*)} - (x(t'')-x(t'))\Bigr], 
\end{aligned}\end{equation}
and $(*)$ is uniformly bounded as long as $\tilde J(x)$ is finite. On the other hand also the finiteness of the right hand side of (\ref{eq:approx_J}) implies uniform boundedness of $(*)$. 
Hence, either both expressions in (\ref{eq:approx_J}) are infinite or we conclude with Lemma \ref{simplJ}  that for an appropriate choice of $t_j^n$,
\begin{align*}
\sup_{t'=t_0<\dots<t_p=t''} &  \sum_{j=1}^p  \Lam^*_{x(t_{j-1}),x(t_j)}\big(t_{j}-t_{j-1}\big)
= \lim_{n\to\infty} \sum_{j=1}^n  \Lam^*_{x(t_{j-1}^n),x(t_j^n)}\big(t_{j}^n-t_{j-1}^n\big)\\
&= \lim_{n\to\infty} \sum_{j=1}^n \big(t_{j}^n-t_{j-1}^n\big)\, x(t_{j}^n)^{\frac{\alpha}{1-\alpha}} \,
\psi\Big(\frac{x(t_{j}^n)-x(t_{j-1}^n)}{t_j^n-t_{j-1}^n}\Big)\\
&=\sup_{t'=t_0<\dots<t_p=t''} \sum_{j=1}^p \big(t_{j}-t_{j-1}\big)\, x(t_{j})^{\frac{\alpha}{1-\alpha}} \,
\psi\Big(\frac{x(t_{j})-x(t_{j-1})}{t_j-t_{j-1}}\Big).
\end{align*}
This expression easily extends to formula  (\ref{eq:approx_J}).
\end{proof}

\begin{lemma}\label{Jgood}
The level sets of~$J$ are compact in $\cI[0,\infty)$.
\end{lemma}

\begin{proof}
We have to verify the assumptions of the Arzel\`a-Ascoli theorem. Fix  $\delta\in(0,1)$, $t\geq 0$, and a function 
$x\in \cI[0,\infty)$ with finite rate $J$. We choose $\delta'\in(0,\delta)$ with $x_{t+\delta '}=\frac12(x_t+x_{t+\delta})$, denote $\eps=x_{t+\delta}-x_t$, and observe that
\begin{align*}
J(x)&\geq \int_{t}^{t+\delta} x_s^{\frac\alpha{1-\alpha}} [1-\dot x_s+\dot x_s \log \dot x_s]\,ds\\
&\geq (\delta-\delta') \Bigl(\frac\eps2\Bigr)^{\frac\alpha{1-\alpha}} \int_{t+\delta '}^{t+\delta} [1-\dot x_s+\dot x_s \log \dot x_s]\,\frac{ds}{\delta-\delta'}.
\end{align*}
Here we used that $x_s\geq \eps/2$ for $s\in [t+\delta',t+\delta]$.
Next, we apply Jensen's inequality to the  convex function $\psi$ to deduce that
\begin{align*}
J(x)&\geq (\delta-\delta') \Bigl(\frac\eps2\Bigr)^{\frac\alpha{1-\alpha}} \psi\Bigl(\frac1{\delta-\delta'} \frac\eps 2\Bigr).
\end{align*}
Now assume that $\frac \eps2 \geq \delta$. Elementary calculus yields
$$
J(x)\geq \delta \Bigl(\frac\eps2\Bigr)^{\frac\alpha{1-\alpha}} \psi\Bigl(\frac1{\delta} \frac\eps 2\Bigr)\geq   \Bigl(\frac\eps2\Bigr)^{\frac1{1-\alpha}} \log\frac\eps{2e\delta}.
$$
If we additionally assume $\eps \geq 2e\delta^{\frac 12}$, then we get
$(J(x)/\log \delta^{-\frac12})^{1-\alpha} \geq   \eps$. Therefore, in general
$$x_{t+\delta}-x_t \leq 
\max\Bigl(2\Bigl(\frac{J(x)}{\log \delta^{-\frac12}}\Bigr)^{1-\alpha},  2e\delta^{\frac12}\Bigr).$$
Hence the level sets are uniformly equicontinuous. As $x_0=0$ for all $x\in\cI[0,\infty)$ this implies 
that the level sets are uniformly bounded on compact sets, which finishes the proof. 
\end{proof}

We now improve our large deviation principle to the topology of
uniform convergence on compact sets, which is stronger than the
topology of pointwise convergence. To this end we introduce,
for every $m\in\IN$, a mapping $f_m$ acting on functions 
$x\colon [0,\infty) \to \IR$ by
\begin{equation}\label{fmdef}
f_m(x)_t=  x_{t_{j}} \qquad \mbox{ if }  t_j:=\sfrac jm \leq t
< \sfrac{j+1}m =:t_{j+1}.\\[2mm]
\end{equation}

\begin{lemma}\label{uniapp}
For every $\delta>0$ and $T>0$, we have
$$\lim_{m\to\infty} \limsup_{\kappa\uparrow\infty} \frac1{a_\kappa}
\log \IP\big( \sup_{0\le t \le T} \big| f_m\big(\sfrac1\kappa Z_{\kappa\,\cdot} \big)_t - 
\sfrac1\kappa Z_{\kappa t} \big| > \delta \big) = -\infty.$$
\end{lemma}

\begin{proof}
Note that
$$\begin{aligned} 
\IP\big( \sup_{0\le t \le T} \big| f_m\big(\sfrac1\kappa Z_{\kappa\,\cdot} \big)_t - 
\sfrac1\kappa Z_{\kappa t} \big| \geq \delta \big)
& \leq \sum_{j=0}^{Tm} \IP\big( \sfrac1\kappa\,Z_{\kappa t_{j+1}} - \sfrac1\kappa\,Z_{\kappa t_{j}} \geq \delta  \big).
\end{aligned}$$
By Lemma~\ref{weakldp} we have
$$\limsup_{\kappa\uparrow\infty} \frac1{a_\kappa} \, \log \IP\big( 
\sfrac1\kappa\,Z_{\kappa t_{j+1}} - \sfrac1\kappa\,Z_{\kappa t_{j}} \geq \delta \big)
\le \inf\big\{ J(x) \colon x_{t_{j+1}}-x_{t_j} \geq \delta \big\},$$
and, by Lemma~\ref{Jgood}, the right hand side diverges to infinity, uniformly in~$j$,
as~$m\uparrow\infty$.
\end{proof}

\begin{proof}[ of the first large deviation principle in Theorem~\ref{LDP}]
We apply \cite[Theorem 4.2.23]{DemZei98}, which allows to transfer
the large deviation principle from the topological Hausdorff space
of increasing functions with the topology of pointwise convergence, 
to the metrizable space ${\mathcal I}[0,\infty)$ by means of the sequence 
$f_m$ of continuous mappings approximating the identity. Two conditions need
to be checked: On the one hand, using the equicontinuity of the sets $\{I(x)\leq \eta\}$ 
established in Lemma~\ref{Jgood}, we easily obtain
$$\limsup_{m\to\infty} \sup_{J(x)\leq \eta} d\big( f_m(x), x
\big) =0,$$ for every $\eta>0$, where $d$ denotes a suitable metric
on ${\mathcal I}[0,\infty)$. On the other hand, by Lemma~\ref{uniapp}, we have 
that $(f_m(\frac1\kappa Z_{\kappa\,\cdot}))$ are a family of exponentially
good approximations of $(\frac1\kappa Z_{\kappa\,\cdot})$.
\end{proof}
\bigskip

The proof of the second large principle can be done from first
principles.
\bigskip

\begin{proof}[ of the second large deviation principle in Theorem~\ref{LDP}]
For the \emph{lower} bound observe that, for any~$T>0$ and $\eps>0$,
$$\IP\big( \sup_{0\le t\le T} |{\sfrac1\kappa Z_{\kappa t}} - (t-a)_+ | <\eps \big)\geq \IP\big( Z_{\kappa a}=0 \big) \, 
\IP\big( \sup_{a\le t \le T}|{\sfrac1\kappa (Z_{\kappa t}-Z_{\kappa a}) - (t-a)}| <\eps \big),$$
and recall that the first probability on the right hand side is
$\exp\{-\kappa \, a f(0)\}$ and the second  converges to one, by the
law of large numbers. %\smallskip
For the \emph{upper} bound note first that, by the first large
deviation principle, for any $\eps>0$ and closed set
$A\subset\{J(x)>\eps\}$,
$$\limsup_{\kappa\uparrow\infty} \frac1\kappa\, \log\IP\big( {\sfrac1\kappa Z_{\kappa\,\cdot}}
\in A\big)  = - \infty.$$ Note further that, for any $\delta>0$ and~$T>0$,
there exists $\eps>0$ such that $J(x)\le\eps$ implies
$\sup_{0\le t\le T}|x-y|<\delta$, where $y_t=(t-a)_+$ for some $a\in[0,T]$. Then,
for $\theta<f(0)$,
$$\begin{aligned}
\IP\big( \sup_{0\le t\le T} |\sfrac1\kappa Z_{\kappa t}-y| \leq \delta \big) & \leq 
\IP\big( Z_{\kappa a} \leq \delta\kappa \big)= \IP\big( T[0,\kappa\delta] \geq
\kappa a\big) \leq  e^{-\kappa a \theta}\,  \prod_{\Phi(j) \le
\kappa\delta}
\IE\exp\big\{ \theta S_j \big\} \\
& = e^{-\kappa a \theta}\,  \exp \sum_{\Phi(j) \le \kappa\delta}
\log \frac{1}{1-\frac\theta{f(j)}} \, ,
\end{aligned}$$
and the result follows because the sum on the right is bounded by a
constant multiple of $\kappa\delta$.
\end{proof}

\subsection{The moderate deviation principle}

Recall from the beginning of Section~\ref{sec42} that it is sufficient to show 
Theorem~\ref{MDP} for the approximating process $Z$ defined in~\eqref{Zdef}. 
We initially include the case~$c=\infty$ in our consideration, and abbreviate
$$b_\kappa:= a_\kappa \, \kappa^{\frac{2\alpha-1}{1-\alpha}} \, \bar{\ell}(\kappa)
\ll \kappa^{\frac{\alpha}{1-\alpha}} \, \bar{\ell}(\kappa),$$ so
that we are looking for a moderate deviation principle with speed~$a_\kappa b_\kappa$.

\begin{lemma}\label{MDPocc}
Let $0\leq u<v$, suppose that  $f$ and $a_\kappa$ are as in
Theorem~\ref{MDP} and define
$$\cI_{[u,v)}=\int_{u}^v s^{-\frac\alpha{1-\alpha}}\, ds= \sfrac{1-\alpha}{1-2\alpha}\, \Big(v^{\frac{1-2\alpha}{1-\alpha}}-u^{\frac{1-2\alpha}{1-\alpha}}\Big).$$
Then the family
$$\bigg( \frac{T{[\kappa u,\kappa v)}-\kappa(v-u))}
{a_\kappa} \bigg)_{\kappa>0}$$ satisfies a large deviation principle
with speed $(a_\kappa b_\kappa)$ and rate function
$$I_{[u,v)}(t)=\begin{cases}
\frac1{2\cI_{[u,v)}} \, t^2 & \text { if } u>0 \mbox{ or }
t\leq \sfrac1c
\, \cI_{[0,v)} \,f(0),\\[2mm]
\sfrac1c\,f(0)\, t-\frac12 \cI_{[0,v)} (\sfrac1c\,f(0))^2&\text{ if }u=0 \mbox{ and } t\geq \sfrac1c
\, \cI_{[0,v)} \,f(0).
\end{cases}$$
\end{lemma}

\begin{proof} Denoting by $\Lam_\kappa$ the logarithmic moment generating function of
$b_\kappa \, ( T[\kappa u,\kappa v)-\kappa(v-u) )$, observe that
\begin{align}
\Lam_\kappa(\th)&= \log \IE \exp\big\{ \th b_{\kappa} \,
\big(T[\kappa u,\kappa v)-\kappa(v-u) \big) \big\}  =
\!\!\!\sum_{w\in\IS\cap [\kappa u,\kappa v)} \!\!\!\!\!\!\log
\IE \exp\big\{ \th b_{\kappa} \,  \big(T[w] \big) \big\}   -\th\kappa b_{\kappa}(v-u) \nonumber\\[2mm]
& = \sum_{w\in\IS\cap [\kappa u,\kappa v)} \xi\Big(\frac{\th
b_\kappa}{\bar f(w)}\Big) -\th \kappa b_\kappa \,(v-u) =
\int_{\II_\kappa} \bar f(w) \xi \Big(\frac{\th b_\kappa}{\bar
f(w)}\Big)\, dw -\th \kappa b_\kappa\, (v-u), \label{eq0805-1}
\end{align}
where $\II_\kappa=\{w\geq 0: \iota (w)\in[\kappa u,\kappa v)\}$ and $\iota(w)=\max \IS\cap [0,w]$.
Since  $\kappa u\leq \inf \,\II_\kappa< \kappa u+(f(0))^{-1}$ and  $\kappa v\leq \sup \II_\kappa<\kappa v+(f(0))^{-1}$ we get
\begin{align}\label{eq1025-1}
\Bigl|\Lam_\kappa(\th)- \int_{\II_\kappa}  \bigl[ \bar f(w) \xi \Big(\frac{\th b_\kappa}{\bar f(w)}\Big)
-\th b_\kappa\bigr]\, dw
\Bigr| \leq \frac{2\th b_\kappa}{f(0)}.
\end{align}
Now focus on the case $u>0$. A Taylor approximation gives $\xi(w)= w+\frac12 (1+o(1)) w^2$, as $w\downarrow 0$.
By dominated convergence,
\begin{align*}\begin{split}%\label{eq0805-2}
\int_{\II_\kappa}  \bigl[ \bar f(w) \xi \Big(\frac{\th b_\kappa}{\bar f(w)}\Big)-\th b_\kappa\bigr]\, dw
&\sim \frac12 \, \int_{\II_\kappa}  \frac 1 {\bar f(w)}\,dw \times \th^2b_\kappa^2\\
&\sim  \frac12 \, \frac{\kappa^{\frac{1-2\alpha}{1-\alpha}}}{\bar
\ell(\kappa)} \int_u^v w^{-\frac\alpha{1-\alpha}}\,dw\times \th^2
b_\kappa^2 \\ & =a_\kappa b_\kappa \, \frac12 \, \cI_{[u,v)}
\,\th^2\, .
\end{split}
\end{align*}
Together with (\ref{eq1025-1}) we arrive at
$$\Lam_\kappa(\th)\sim a_\kappa b_\kappa \, \frac12 \, \cI_{[u,v)}  \th^2.$$
Now the G\"artner-Ellis theorem implies that the family $((T[\kappa u,\kappa v)-\kappa(v-u))/a_\kappa)$ satisfies a large deviation
principle with speed $(a_\kappa b_\kappa)$ having as rate function the
Fenchel-Legendre transform of $\frac12 \cI_{[u,v)} \th^2$ which is $I_{[u,v)}$.
\medskip

Next, we look at the case~$u=0$. If $\th\geq \sfrac1c\,f(0)$ then $\Lam_\kappa(\th)=\infty$ for all $\kappa>0$,
so assume the contrary. The same Taylor expansion as above now gives
$$\bar f(w) \xi\Bigl(\frac{\th b_\kappa}{\bar f(w)}\Bigr)-\th b_\kappa
\sim \frac12 \frac{\th^2 b_\kappa^2}{\bar f(w)}$$
as $w \uparrow \infty$.
In particular, the integrand in \eqref{eq0805-1}
is regularly varying with index $-\frac{\alpha}{1-\alpha}>-1$ and we get from
Karamata's theorem, see e.g. \cite[Theorem 1.5.11]{BiGoTeu87},
 that
\begin{align}\label{eq0809-3}
\Lam_\kappa(\th)\sim \frac12\, \th^2 b_\kappa^2\, \frac{\kappa^{\frac{1-2\alpha}{1-\alpha}}}
{\bar \ell(\kappa)} \int _0^v s^{-\alpha/(1-\alpha)} \,ds=a_\kappa b_\kappa\,\frac12 \, \cI_{[0,v)} \, \th^2.
\end{align}
Consequently,
$$\lim_{\kappa\to\infty} \frac1{a_\kappa b_\kappa} \Lam_\kappa(\th)=\begin{cases}
\frac12\, \cI_{[0,v)} \th^2 &\text{ if  }\th<\sfrac1c\,f(0),\\
\infty& \text{ otherwise.}
\end{cases}
$$
The Legendre transform of the right hand side is
$$
I_{[0,v)}(t)=\begin{cases}
\frac1{2\cI_{[0,v)}} t^2 & \text { if } t\leq\sfrac1c
\, \cI_{[0,v)} \,f(0),\\
\sfrac1c\,f(0)\, t-\frac12 \cI_{[0,v)} \big(\sfrac1c\,f(0)\big)^2&\text{ if }t\geq \sfrac1c
\, \cI_{[0,v)} \,f(0).
\end{cases}
$$
Since $I_{[0,v)}$ is not strictly convex the G\"artner-Ellis Theorem does not imply the full large deviation principle. It remains to prove the lower bound for open sets $(t,\infty)$ with $t\geq \sfrac1c\,\cI_{[0,v)} f(0)$. Fix $\eps\in(0,u)$  and note that, for sufficiently large~$\kappa$,
\begin{align*}
\IP\big((T{[\kappa u,\kappa v)}-\kappa v)/a_\kappa>t\big) & \geq \IP\big((T{[\kappa \eps,\kappa v)}-\kappa (v-\eps))/a_\kappa>\sfrac1c\,\cI_{[0,v)} f(0)\big)\\
& \quad\times \underbrace{\IP\big((T{(0,\kappa \eps)}-\kappa \eps)/a_\kappa>-\eps\big)}_{\to1} \, \IP\big(T[0]/a_\kappa>t-\sfrac1c\, \cI_{[0,v)} f(0)+\eps\big).
\end{align*}
so that by the large deviation principle for $((T_{[\kappa \eps,\kappa v)}-\kappa (v-\eps))/a_\kappa)$ and the exponential distribution it follows that
$$
\liminf _{\kappa\to \infty} \frac1{a_\kappa b_\kappa}
\log \IP\big((T{[0,\kappa v)}-\kappa v)/a_\kappa>t\big)
\geq -\frac{(\sfrac1c\,\cI_{[0,v)} f(0))^2}{2\cI_{[\eps,v)}} -(t-\sfrac1c\,\cI_{[0,v)} f(0)+\eps)
\sfrac1c\,f(0).
$$
Note that the right hand side converges to $-I_{[0,v)}(t)$ when letting $\eps$ tend to zero.
This establishes the full large deviation principle for $((T{[0,\kappa v)}-\kappa v )/a_\kappa)$.
\end{proof}

%\begin{remark}
%An extension of the statement to the case  $\alpha=1/2$ is also possible. However, in this
%case the choice of the slowly varying function $\bar \ell$ plays an important role.
%In particular, if $u=0$, one can not apply Karamata's theorem to deduce \eqref{eq0809-3} and
%one needs to analyse the asymptotic behaviour of $\int_0^\cdot \frac1{\bar f(s)}\,ds$.
%\end{remark}
\pagebreak[3]

We continue the proof of Theorem~\ref{MDP} with a finite-dimensional
moderate deviation principle, which can be derived from
Lemma~\ref{MDPocc}.

\begin{lemma}
Fix $0=t_0<t_1<\cdots<t_p$. Then the vector
$$\Big( \sfrac1{a_\kappa} \big(Z_{\kappa t_j} - \kappa t_j\big) \, :\, j\in\{1,\ldots,p\}\Big)$$
satisfies a large deviation principle in $\IR^p$ with speed $a_\kappa b_\kappa$ and rate function
$$I(a_1,\ldots,a_p)=\sum_{j=1}^p I_{[t_{j-1},t_j)}(a_{j-1}-a_{j}), \qquad \mbox{ with }a_0:=0\, .$$
\end{lemma}

\begin{proof}
We note that, for $-\infty\leq a^{\ssup j}<b^{\ssup j}\leq\infty$, we
have (interpreting conditions on the right as void, if they involve
infinity)
$$\begin{aligned}
\IP\big( a^{\ssup j} & \, a_\kappa \leq Z_{\kappa t_j} - \kappa t_j< b^{\ssup j} \,  a_\kappa
\mbox{ for all }j \big) \\
& = \IP\big( T[0,\kappa t_j+ a_\kappa a^{\ssup j}) \leq \kappa\, t_j,
T[0,\kappa t_j+ a_\kappa b^{\ssup j}) > \kappa\, t_j\mbox{ for all }j \big)\, .
\end{aligned}$$
To continue from here we need to show that the random variables $T[0,\kappa t+ a_\kappa b)$ and
$T[0,\kappa t)+ a_\kappa b$ are exponentially equivalent in the sense that
\begin{align}\label{expeq} \lim_{\kappa\to\infty}
a_\kappa^{-1}b_\kappa^{-1} \log \IP\big( \big|T[0,\kappa t+ a_\kappa b)-T[0,\kappa t)- a_\kappa b\big|
> a_\kappa \eps\big)=-\infty\, .
\end{align}
Indeed, first let $b>0$. As in Lemma~\ref{MDPocc}, we see that for
any $t\geq 0$ and $\theta\in\IR$,
\begin{align}
a_\kappa^{-1} b_\kappa^{-1}\,\log \IE \exp \big\{ \theta b_\kappa \big( T[\kappa t, \kappa t + a_\kappa b)
-   a_\kappa b \big) \big\} \longrightarrow 0,
\end{align}
Chebyshev's inequality gives, for any $A>0$,
$$\begin{aligned} \IP\big( T[0, & \kappa t+ a_\kappa b)-T[0,\kappa t)- a_\kappa b> a_\kappa \eps\big) \\
& \leq  e^{-A a_\kappa b_\kappa}\, \IE \exp \big\{ \sfrac A\eps\,
b_\kappa \big( T[\kappa t, \kappa t + a_\kappa b) -   a_\kappa b
\big) \big\}.
\end{aligned}$$
A similar estimate can be performed for $\IP( T[0, \kappa t+
a_\kappa b)-T[0,\kappa t)- a_\kappa b< - a_\kappa \eps)$, and the
argument also extends to the case~$b<0$. From this~\eqref{expeq}
readily follows.\smallskip

Using Lemma~\ref{MDP} and independence, we obtain a large deviation principle for the vector
$$\Big( \sfrac1{a_\kappa} \big(T[\kappa t_{j-1},\kappa t_j) -\kappa\, (t_j-t_{j-1})\big) \, : \, j\in\{1,\ldots,p\}\Big),$$
with rate function
$$I_1(a_1,\ldots,a_p)=\sum_{j=1}^p I_{[t_{j-1},t_j)}(a_j).$$
Using the contraction principle, we infer from this a large deviation principle
for the vector
$$\Big( \sfrac1{a_\kappa} \big(T[0,\kappa t_j) -\kappa\, t_j\big) \, : \, j\in\{1,\ldots,p\}\Big)$$
with rate function
$$I_2(a_1,\ldots,a_p)=\sum_{j=1}^p I_{[t_{j-1},t_j)}(a_j-a_{j-1})\, .$$
Combining this with \eqref{expeq} we obtain that
\begin{align*}
a_\kappa^{-1}b_\kappa^{-1} & \log \IP\big( T[0,\kappa t_j+ a_\kappa a^{\ssup j}] < \kappa\, t_j,
T[0,\kappa t_j+ a_\kappa b^{\ssup j}] > \kappa\, t_j\mbox{ for all }j \big)\\
& \sim a_\kappa^{-1}b_\kappa^{-1} \log
\IP\big( -  a_\kappa b^{\ssup j} < T[0,\kappa t_j] -\kappa\, t_j<  - a_\kappa a^{\ssup j} \mbox{ for all }j \big),
\end{align*}
and (observing the signs!) the required large deviation principle.
\end{proof}
\smallskip

We may now take a projective limit and arrive at a large deviation
principle in the space~${\mathcal P}(0,\infty)$ of functions
$x\colon(0,\infty)\to\IR$ equipped with the topology of pointwise
convergence.

\begin{lemma}\label{pointwise}
The family of functions
$$\Big( \sfrac1{a_\kappa} \big(Z_{\kappa t} - \kappa t\big) \colon t>0 \Big)_{{\kappa>0}}$$
satisfies a large deviation principle in the space~${\mathcal
P}(0,\infty)$, with speed $a_\kappa b_\kappa$ and rate function
$$I(x)= \left\{ \begin{array}{ll}\frac 12 \int_0^\infty (\dot{x}_t)^2 \, t^{\frac{\alpha}{1-\alpha}} \, dt
- \sfrac1c\,f(0)\,x_0 & \mbox{ if $x$ is absolutely continuous and $x_0\leq 0$.}\\
\infty & \mbox{ otherwise.} \end{array} \right.$$
\end{lemma}

\begin{proof}
Observe that the space of functions equipped with the topology of
pointwise convergence can be interpreted as the projective limit of
$\IR^p$ with the canonical projections given by
$\pi(x)=(x(t_1),\ldots,x(t_p))$ for $0<t_1<\ldots<t_p$. By the
Dawson-G\"artner theorem, we obtain a large deviation principle with
rate function
$$\begin{aligned}
\tilde I(x) & = \sup_{0<t_1<\ldots<t_p} \sum_{j=2}^p
I_{[t_{j-1},t_j)} \big(x_{t_{j-1}}-x_{t_j}\big) +
I_{[0,t_1)}\big(-x_{t_1}\big).
\end{aligned}$$
Note that the value of the variational expression is nondecreasing,
if additional points are added to the partition. We first fix
$t_1>0$ and optimize the first summand independently. Observe that
$$\begin{aligned}
\sup_{t_1<\ldots<t_p} & \sum_{j=2}^p I_{[t_{j-1},t_j)}
\big(x_{t_{j-1}}-x_{t_j}\big)  = \sfrac{1}{2} \,
\sup_{t_1<t_2<\ldots<t_p}  \sum_{j=2}^p
\frac{\big(x_{t_{j}}-x_{t_{j-1}}\big)^2}{\frac{1-\alpha}{1-2\alpha}\big(
t_j^{\frac{1-2\alpha}{1-\alpha}}-t_{j-1}^{\frac{1-2\alpha}{1-\alpha}}\big)}.
\end{aligned}$$
Recall that
$$\big(t_j-t_{j-1}\big)\, t_{j}^{\frac{-\alpha}{1-\alpha}}\leq
\sfrac{1-\alpha}{1-2\alpha}\big(t_j^{\frac{1-2\alpha}{1-\alpha}}-t_{j-1}^{\frac{1-2\alpha}{1-\alpha}}\big)
\leq\big(t_j-t_{j-1}\big)\, t_{j-1}^{\frac{-\alpha}{1-\alpha}}\, .$$
Hence we obtain an upper and [in brackets] lower bound of
$$\begin{aligned}
\sfrac{1}{2} \, \sup_{t_1<t_2<\ldots<t_p}  \sum_{j=2}^p \Big(
\frac{x_{t_{j}}-x_{t_{j-1}}}{t_j-t_{j-1}}\Big)^2 \,
t_{j[-1]}^{\frac{\alpha}{1-\alpha}} \,\big( t_j-t_{j-1}\big) .
\end{aligned}$$
It is easy to see that (using arguments analogous to those given in the last step in the proof of 
the first large deviation principle) that this is $+\infty$ if $x$ fails to be
absolutely continuous, and otherwise it equals
$$\sfrac 12 \, \int_{t_1}^\infty (\dot{x}_t)^2 \, t^{\frac{\alpha}{1-\alpha}} \, dt.$$
In the latter case we have
$$\tilde I(x) = \lim_{t_1\downarrow 0} \, \sfrac 12 \, \int_{t_1}^\infty
(\dot{x}_t)^2 \, t^{\frac{\alpha}{1-\alpha}} \, dt +
I_{[0,t_1)}\big(-x_{t_1}\big).$$ If $x_0>0$ the last summand
diverges to infinity. If $x_0=0$ and the limit of the integral is
finite, then using Cauchy-Schwarz,
$$I_{[0,t_1)}\big(-x_{t_1}\big) \leq \sfrac12 \cI^{-1}_{[0,t_1)}
\, \Big|\int_0^{t_1} {\dot x}_t \,  dt\Big|^2 \leq\sfrac12
\int_0^\eps ({\dot x}_t)^2 t^{\frac{\alpha}{1-\alpha}}\,dt,$$ hence
it converges to zero. If $x_0<0$,
$$\lim_{t_1\downarrow 0} I_{[0,t_1)}\big(-x_{t_1}\big)
= \lim_{t_1\downarrow 0} - \sfrac1c\,f(0)\,
x_{t_1}+\sfrac{1-\alpha}{1-2\alpha} \,
t_1^{\frac{\alpha}{1-\alpha}}(\sfrac1c\,f(0))^2 = - \sfrac1c\,f(0)\,
x_0,$$ as required to complete the proof.\end{proof}
\bigskip

\begin{lemma}\label{AA}
If $c<\infty$, the function~$I$ is a good rate function
on~${\mathcal L}(0,\infty)$.
\end{lemma}

\begin{proof}
Recall that, by the Arzel\`a-Ascoli theorem, it suffices to show
that for any $\eta>0$ the family $\{x \colon I(x)\leq \eta\}$ is
bounded and equicontinuous on every compact subset of~$(0,\infty)$.
\smallskip

Suppose that $I(x)\leq \eta$ and $0<s<t$. Then, using Cauchy-Schwarz
in the second step,
\begin{align*}
|x_t-x_s| & = \Big| \int_s^t {\dot x}_u\, du \Big| \leq \int_s^t
|{\dot x}_u|\, u^{\frac{\alpha}{2(1-\alpha)}} \,
u^{-\frac{\alpha}{2(1-\alpha)}} \, du\\
& \leq \Big( \int_s^t ({\dot x}_u)^2\, u^{\frac{\alpha}{1-\alpha}}
\, du\Big)^{\frac12} \, \Big( \int_s^t u^{-\frac{\alpha}{1-\alpha}}
\, du\Big)^{\frac12} \leq \sqrt{\eta\,\sfrac{1-\alpha}{1-2\alpha}}\,
\big(t^{\frac{1-2\alpha}{1-\alpha}}-s^{\frac{1-2\alpha}{1-\alpha}}\big)^{\sfrac12},
\end{align*}
which proves equicontinuity. The boundedness condition follows from
this, together with the observation that $0\geq x_0\geq
-c\eta/f(0)$.
\end{proof}
\medskip

To move our moderate deviation principle to the topology of
uniform convergence on compact sets, recall the definition
of the mappings~$f_m$ from~\eqref{fmdef}. We abbreviate
$$\bar Z^{\ssup \kappa}:=
\Big( \sfrac1{a_\kappa} \big(Z_{\kappa t} - \kappa
t\big) \, :\, t>0 \Big).$$

\begin{lemma}\label{ea}
$(f_m(\bar Z^{\ssup \kappa}))_{m\in\IN}$  are exponentially good
approximations of $(\bar Z^{\ssup \kappa})$ on ${\mathcal
L}(0,\infty)$.
\end{lemma}
\medskip

\begin{proof}
We need to verify that, denoting by $\|\,\cdot\,\|$ the supremum
norm on any compact subset of~$(0,\infty)$, for every $\delta>0$,
$$\lim_{m\to\infty} \limsup_{\kappa\to\infty} a_\kappa^{-1}
b_\kappa^{-1} \log \IP\big( \|\bar Z^{\ssup \kappa}-f_m(\bar
Z^{\ssup \kappa}) \| > \delta \big) = - \infty.$$ The crucial step
is to establish that, for sufficiently large $\kappa$, for all
$j\geq 2$,
\begin{equation}\label{vare}\IP\Big(
\sup_{t_{j-1}\!\!\leq t< t_j} \big| \bar Z^{\ssup \kappa}_t-\bar
Z^{\ssup \kappa}_{t_{j-1}} \big| > \delta \Big) \leq 2 \, \IP\big(
\big| \bar Z^{\ssup \kappa}_{t_j}- \bar Z^{\ssup
\kappa}_{t_{j-1}}\big|
> \sfrac\delta 2 \big).
\end{equation}
To verify~\eqref{vare} we use the stopping time $\tau:=\inf\{t\geq
t_{j-1} \colon |\bar Z^{\ssup \kappa}_t-\bar Z^{\ssup
\kappa}_{t_{j-1}}| > \delta\}$. Note that
$$\begin{aligned}
\IP\big( \big| & \bar Z^{\ssup \kappa}_{t_j}-\bar Z^{\ssup
\kappa}_{t_{j-1}} \big| > \sfrac\delta2 \big)\\ & \geq  \IP\Big(
\sup_{t_{j-1}\le t< t_j}\big| \bar Z^{\ssup \kappa}_t-\bar
Z^{\ssup \kappa}_{t_{j-1}} \big| > \delta \Big) \, \IP\big( \big|
\bar Z^{\ssup \kappa}_{t_j}-\bar Z^{\ssup \kappa}_{t_{j-1}} \big| >
\sfrac\delta2 \, \big| \, \tau \leq t_j \big) ,
\end{aligned}$$
and, using Chebyshev's inequality in the last step,
$$\begin{aligned}
\IP\big( \big| \bar Z^{\ssup \kappa}_{t_j}-\bar Z^{\ssup
\kappa}_{t_{j-1}} \big| > \sfrac\delta2 \, \big| \, \tau \leq t_j
\big) & \geq \IP\big( \big| \bar Z^{\ssup \kappa}_{t_j}-\bar
Z^{\ssup \kappa}_{\tau} \big| \leq \sfrac\delta2 \, \big| \, \tau
\leq t_j \big)\geq 1 - \sfrac{4}{\delta^2} \, {\rm Var}\big( \bar
Z^{\ssup \kappa}_{t_j-t_{j-1}} \big).
\end{aligned}$$
As this variance is of order
$a_\kappa^{-2}\,\kappa^{\frac{1-2\alpha}{1-\alpha}}\,
{\bar\ell}(\kappa)^{-1} \rightarrow 0$, the right hand side exceeds
$\frac12$ for sufficiently large~$\kappa$, thus
proving~\eqref{vare}.
\smallskip

With \eqref{vare}  at our disposal, we observe that, for some
integers $n_1\geq n_0 \geq 2$ depending only on~$m$ and the chosen
compact subset of $(0,\infty)$,
$$\begin{aligned}
\IP\big( \|\bar Z^{\ssup \kappa}-f_m(\bar Z^{\ssup \kappa}) \| >
\delta \big) & \leq \sum_{j=n_0}^{n_1} \IP\Big(
\sup_{t_{j-1}\le t< t_j} \big|\bar Z^{\ssup \kappa}_{t}-\bar
Z^{\ssup \kappa}_{t_{j-1}} \big| > \delta \Big) \\
& \leq 2 \!\sum_{j=n_0}^{n_1} \IP\big( \big|\bar Z^{\ssup
\kappa}_{t_j}-\bar Z^{\ssup \kappa}_{t_{j-1}} \big| > \sfrac\delta2
\big) .
\end{aligned}$$
Hence, we get $$\begin{aligned} \limsup_{\kappa\to\infty}
a_\kappa^{-1} b_\kappa^{-1} \log \IP\big( \|\bar Z^{\ssup
\kappa}-f_m(\bar Z^{\ssup \kappa}) \| > \delta \big) & \le
-\inf_{j=n_0}^{n_1} I_{[t_{j-1},t_j)}(\sfrac\delta2),
\end{aligned}$$
and the right hand side can be made arbitrarily small by making
$m=\frac1{t_j-t_{j-1}}$ large.
\end{proof}
\medskip

\begin{proof}[ of Theorem~\ref{MDP}]
We apply \cite[Theorem 4.2.23]{DemZei98} to transfer
the large deviation principle from the topological Hausdorff space
${\mathcal P}(0,\infty)$ to the metrizable space ${\mathcal
L}(0,\infty)$ using the sequence $f_m$ of continuous functions.
There are two conditions to be checked for this, on the one hand
that $(f_m(\bar Z^{\ssup \kappa}))_{m\in\IN}$  are exponentially
good approximations of $(\bar Z^{\ssup \kappa})$, as verified in
Lemma~\ref{ea}, on the other hand that
$$\limsup_{m\to\infty} \sup_{I(x)\le\eta} d\big( f_m(x), x
\big) =0,$$ for every $\eta>0$, where $d$ denotes a suitable metric
on ${\mathcal L}(0,\infty)$. This follows easily from the
equicontinuity of the set $\{I(x)\leq \eta\}$ established in
Lemma~\ref{AA}. Hence the proof is complete.~\end{proof}

\section{The vertex with maximal indegree}\label{sec6}

In this section we prove Theorem~\ref{main} and Theorem~\ref{WL}.

{\subsection{Strong and weak preference: Proof of Theorem~\ref{main}}}

The key to the proof is Proposition~\ref{p:A_s} which shows that, in the strong
preference case, the degree of a fixed vertex can only be surpassed by a finite 
number of future vertices. The actual formulation of the result also contains
a useful technical result for the weak preference case.

Recall that $\vphi_t=\int_0^t \frac1{\bar f(v)}\, dv$, and let
$$t(s)=\sup\{t\in\IS \colon 4 \vphi_t\leq s\}, \quad\mbox{ for } s\ge 0.$$
Moreover, we let $\vphi_\infty=\lim_{t\to\infty} \vphi_t$, which is finite
exactly in the strong preference case. In this case $t(s)=\infty$ eventually.

\begin{prop}\label{p:A_s}
For any fixed $\eta>0$, almost surely only finitely many of the events
$$A_s:=\big\{\exists t'\in [s,t(s))\cap\IT: Z[s,t']\geq t'-\eta\big\}, \quad\mbox{ for }s\in\IT,$$
occur.
\end{prop}

For the proof we identify a family of martingales and then apply the concentration
inequality for martingales, Lemma~\ref{l:hoef}. For $s\in\IT$, let 
$(\bar T^s_u)_{u\in\IS}$ be given by $\bar T_u^s=u-T_s[0,u)$, where $T_s[u,v)$
is the time spent by the process $Z[s,\,\cdot\,]$ in the interval $[u,v)$.

The following lemma is easy to verify.

\begin{lemma}\label{l:var}
Let $(t_i)_{i\iZ_+}$ be a strictly increasing sequence of nonnegative numbers with $t_0=0$ and 
$\lim_{i\to\infty} t_i=\infty$. Moreover, assume that $\lam>0$ is fixed such that
$\lam\,\Delta t_i := \lam \,(t_i-t_{i-1})\leq 1$, for all $i\iN$, and consider a discrete 
random variable $X$ with
$$\IP\big(X=t_i\big)= \lam \Delta t_i\, \prod_{j=1}^{i-1} (1-\lam \Delta t_j)  \qquad \mbox{ for }i\iN .$$
Then
$$\IE[X]=\frac1\lam \qquad \text{ and } \qquad \var (X)\leq \frac 1{\lam^2}.$$
\end{lemma}

With this at hand, we can identify the martingale property of~$(\bar T^s_u)_{u\in\IS}$.

\begin{lemma}\label{l:T_mart}
For any $s\in\IS$, the process $(\bar T^s_u)_{u\in\IS}$ is a martingale with respect to the natural 
filtration~$(\cG_u)$.  Moreover, for two neighbours $u<u_+$ in $\IS$, one has
$$\var \big(\bar T^s_{u_+}\,|\,\cG_u \big)\leq \frac1{\bar f(u)^2}. $$
\end{lemma}

\begin{proof}
Fix two neighbours $u<u_+$ in $\IS$ and observe that given $\cG_u$ (or given the entry time 
$T_s[0,u)+s$ into state $u$) the distribution of $T_s[u]$ is as in Lemma~\ref{l:var} with 
$\lam=\bar f(u)$. Thus the lemma implies that
$$
\IE[\bar T^s_{u+}|\cG_u]= \bar T^s_u +\frac1{f(u)}- \IE[T_s[u]\,|\,\cG_u]=\bar T^s_u
$$
so that $(\bar T_u^s)$ is a martingale. The variance estimate of Lemma~\ref{l:var} yields 
the second assertion.
\end{proof}

\begin{proof}[ of Proposition~\ref{p:A_s}] 
We fix $\eta\geq 1/f(0)$ and $u_0\in\IS$ with $\bar f(u_0)\geq 2$.
We consider $\IP(A_s)$ for sufficiently large $s\in\IT$. More precisely, 
$s$ needs to be large enough such that $t(s)\geq u_0$ and $s-\eta-u_0\geq \sqrt{s/2}$.
We denote by $\sig$ the first time $t$ in $\IT$ for which $Z[s,t]\geq t-\eta$, if such 
a time exists, and  set $\sig=\infty$  otherwise.

We now look at realizations for which $\sig\in [s,t(s))$ or, equivalently, $A_s$ occurs. We set $\nu=Z[s,\sig]$. Since the jumps of $Z[s,\cdot]$ are bounded by $1/f(0)$ we conclude that
$$
\nu\leq \sig-\eta+1/{f(0)}\leq \sig.
$$
Conversely, $T_s[0,\nu)+s$ is the entry time into state $\nu$ and thus equal to $\sig$; therefore,
$$
\nu=Z[s,\sig]\geq T_s[0,\nu)+s-\eta,
$$
and thus
$\bar T^s_\nu=\nu-T_s[0,\nu)\geq s-\eta.$
Altogether, we conclude that
$$A_s\subset\big\{ \exists u\in[0,t(s))\cap\IS \colon \bar T^s_u\geq s-\eta \big\}.$$
By Lemma \ref{l:T_mart} the process $(\bar T^s_u)_{u\in\IS}$ is a martingale. 
Moreover, for consecutive elements $u<u_+$ of $\IS$ that are larger than $u_0$, one has
$$\var(\bar T^s_{u+}|\cG_u)= \frac1{\bar f(u)^2}, \qquad \bar T^s_{u+}-\bar T^s_{u}\leq \frac1{\bar f(u)}
\leq \frac12, \qquad  \text{ and } \qquad \bar T^s_{u_0}\leq u_0.$$
Now we apply the concentration inequality, Lemma~\ref{l:hoef}, and obtain, writing
$\lam_s=s-\eta-u_0-\sqrt{2\vphi_{t(s)}}\geq 0$, that
\begin{align*}
\IP(A_s) & \leq\IP\big(\sup _{u\in[0,t(s))\cap \IS}  \bar T^s_u \geq s-\eta\big)\\
&\leq\IP\big(\sup _{u\in[u_0,t(s))\cap \IS}  \bar T^s_u-\bar T^s_{u_0} \geq s-\eta-u_0\big)
\leq 2 \exp\Bigl(-\frac{\lam_s^2}{2(\vphi_{t(s)} + \lam_s/6)}\Bigr),
\end{align*}
where we use that
$$\sum_{u\in\IS\cap[0,t(s))} \frac1{\bar f(u)^2} = \vphi_{t(s)}.$$
As $\vphi_{t(s)}\leq s/4$, we obtain $\limsup- \frac1s\, \log \IP(A_s) \geq \frac65.$
Denoting by $\iota(t)= \max[0,t]\cap\IT$, we finally get that
$$
\sum_{s\in\IT} \IP(A_s)  \leq  \int_{0}^\infty e^{s} \, \IP(A_{\iota(s)})\,ds<\infty,
$$
so that by Borel-Cantelli, almost surely, only finitely many of the events $(A_s)_{s\in\IT}$ occur.
\end{proof}
\medskip

\begin{proof}[ of Theorem \ref{main}]
We first consider the weak preference case and fix~$s\in\IT$.
Recall that $(Z[s,t]-(t-s))_{t\ge s}$ and $(Z[0,t]-t)_{t\ge 0}$ are independent and satisfy functional central limit theorems (see Theorem \ref{CLT}).
Thus $(Z[s,t]-Z[0,t])_{t\ge s}$ also satisfies a central limit theorem, i.e.\ an appropriately scaled version converges weakly  to the Wiener process. Since the Wiener process changes its sign almost surely for arbitrarily large times, we conclude that $Z[s,t]$ will be larger, respectively smaller, than $Z[0,t]$ for infinitely many time instances. Therefore, $s$ is not a persistent hub, almost surely. This proves the first assertion.
\medskip

In the strong preference case recall that $\vphi_\infty<\infty$. 
For fixed  $\eta>0$, almost surely, only finitely many of the events $(A_s)_{s\in\IT}$ occur,
by Proposition \ref{p:A_s}. Recalling that $Z[0,t]-t$ has a finite limit, we thus get that almost surely only finitely many degree evolutions overtake the one of the first node. It remains to show that the limit points of $(Z[s,t]-t)$ for varying~$s\in\IT$ are almost surely distinct. 
But this is an immediate consequence of Proposition \ref{prop_ac}.
\end{proof}

{\subsection{The typical evolution of the hub: Proof of Theorem~\ref{WL}}}

From now on we assume that the attachment rule $f$ is regularly varying with index $\alpha<\frac12$, and we represent $f$ and $\bar f$ as
$$f(u)=u^\alpha\ell(u) \quad\text{ and }\quad \bar f(u)=u^{\frac\alpha{1-\alpha}} \bar \ell(u)\qquad \mbox{ for } u>0.$$
Moreover, we fix 
$$a_\kappa=\kappa^{\frac{1-2\alpha}{1-\alpha}} \bar \ell(\kappa)^{-1}.$$
For this choice of $(a_\kappa)$ the moderate deviation principle, Theorem \ref{MDP}, 
leads to the speed $(a_\kappa)$, in other words the magnitude of the deviation and 
the speed coincide. The proof of Theorem~\ref{WL} is based  on the following lemma.
\medskip

\begin{lemma}\label{le0924-2}
Fix $0\leq u<v$ and define $\II_\kappa$ as $\II_\kappa= \IT\cap [a_\kappa u,a_\kappa v)$.
Then, for all $\eps>0$,
$$\lim_{\kappa\to \infty} \IP \Bigl(\max_{s\in \II_\kappa} Z[s,\kappa]  \in \kappa+ a_\kappa\Bigl[-v+\sqrt{\sfrac{2-2\alpha}{1-2\alpha}v}-\eps, -u+\sqrt{\sfrac{2-2\alpha}{1-2\alpha}v}+\eps\Bigr]\Bigr)=1.
$$
\end{lemma}

\begin{proof}
Our aim is to analyze the random variable $\max_{s\in \II_\kappa} Z[s,\kappa]$ for large $\kappa$. We fix $\zeta \geq-u$ and observe that
\begin{align}\begin{split}\label{eq0809-1}
\IP(\max_{s\in\II_\kappa} & Z[s,\kappa]< \kappa +a_\kappa \zeta)= \prod_{s\in \II_\kappa} \IP(Z[s,\kappa]< \kappa +a_\kappa \zeta)\\
&\begin{cases}
\leq \IP(Z[s_\ma,\kappa]< \kappa +a_\kappa \zeta)^{\# \II_\kappa}=\IP(T_{s_\ma}[0,\kappa+a_\kappa\zeta)+s_{\ma}>\kappa)^{\# \II_\kappa}\\
\geq \IP(Z[s_\mi,\kappa]< \kappa +a_\kappa \zeta)^{\# \II_\kappa}=\IP(T_{s_\mi}[0,\kappa+a_\kappa\zeta)+s_{\mi}>\kappa)^{\# \II_\kappa},
\end{cases}\end{split}
\end{align}
where $s_\mi$ and $s_\ma$ denote the minimal and maximal element of $\II_\kappa$.

Next, we observe that $\lim_{\kappa\to\infty} s_\ma/a_\kappa= v$ and $\lim_{\kappa\to\infty} s_\mi/a_\kappa= u$. 
Consequently, we can deduce from the moderate deviation principle, Lemma~\ref{MDPocc}, together with Lemma~\ref{le1025-1}, that
\begin{align}\begin{split}\label{eq0809-2}
\log \IP(T_{s_\ma}[0,\kappa+a_\kappa\zeta)+s_{\ma}\leq \kappa)&=\log \IP\Bigl(\frac{T_{s_\ma}[0,\kappa+a_\kappa\zeta)-\kappa-a_\kappa \zeta}{a_\kappa}\leq -\frac {s_{\ma}}{a_\kappa} -\zeta\Bigr)\\
&\sim -a_\kappa I_{[0,1)}(-v-\zeta)=-a_\kappa \frac12 \frac{1-2\alpha}{1-\alpha} (v+\zeta)^2
\end{split}
\end{align}
and analogously that
$$
\log \IP(T_{s_\mi}[0,\kappa+a_\kappa\zeta)+s_{\mi}\leq \kappa)\sim -a_\kappa \frac12 \frac{1-2\alpha}{1-\alpha} (u+\zeta)^2.
$$

Next we prove that
$\IP(\max _{s\in\II_\kappa} Z[s,\kappa]<\kappa+a_\kappa\zeta)$
tends to $0$ when $\zeta<-v+\sqrt{\frac{2-2\alpha}{1-2\alpha}v}$. 

If $\zeta<-u$, then the statement is trivial since  by the moderate deviation principle
$\IP(Z[s_\mi,\kappa]<\kappa+ a_\kappa \zeta)$
tends to zero. Thus we can assume that $\zeta\geq -u$. By~(\ref{eq0809-1}) one has
$$\IP\big(\max_{s\in\II_\kappa} Z[s,\kappa]<\kappa+a_\kappa\zeta\big)
\leq \exp\bigl\{ \# \II_\kappa \log \big(1-\IP(T_{s_\ma}[0,\kappa+a_\kappa\zeta)+s_\ma \leq \kappa)\big)\bigr\}.
$$
and it suffices to show that the term in the exponential tends to $-\infty$ in order to prove the assertion. The term satisfies
\begin{align*}
\# \II_\kappa \log & \big(1-\IP(T_{s_\ma}[0,\kappa+a_\kappa\zeta)+s_\ma \leq \kappa)\big)\\
&\sim -\#\II_\kappa \,\IP(T_{s_\ma}[0,\kappa+a_\kappa\zeta)+s_\ma \leq \kappa)\\
& = -\exp \Bigl\{a_\kappa\Bigl[ \underbrace{\frac1{a_\kappa}\log \#\II_\kappa +\frac1{a_\kappa}\log \IP(T_{s_\ma}[0,\kappa+a_\kappa\zeta)+s_\ma \leq \kappa)}_{=:c_\kappa}\Bigr]\Bigr\}.
\end{align*}
Since $\frac1{a_\kappa} \log \#\II_\kappa$ converges to $v$, we conclude with (\ref{eq0809-2}) that
$$
\lim_{\kappa\to\infty} c_\kappa= v- \sfrac{1-2\alpha}{2-2\alpha} (v+\zeta)^2.
$$
Now elementary calculus implies that the limit is bigger than $0$ by choice
of~$\zeta$. This implies the first part of the assertion.

It remains to prove that $\IP(\max _{s\in\II_\kappa} Z[s,\kappa]<\kappa+a_\kappa\zeta)$
tends to $1$ for $\zeta>-u+\sqrt{\frac{2-2\alpha}{1-2\alpha}v}$.
Now
$$
\IP\big(\max_{s\in\II_\kappa} Z[s,\kappa]<\kappa+a_\kappa\zeta\big)
\geq \exp\Bigl\{ \#\II_\kappa \log \big(1-\IP(T_{s_\mi}[0,\kappa+a_\kappa\zeta)+s_\mi \leq \kappa)\big)\Bigr\}
$$
and it suffices to show that the expression in the exponential tends to $0$. As above we conclude that
\begin{align*}
\#\II_\kappa &\log  (1-\IP(T_{s_\mi}[0,\kappa+a_\kappa\zeta)+s_\mi \leq \kappa))\\
&\sim  -\exp \bigl(a_\kappa\bigl[ \underbrace{\frac1{a_\kappa}\log \#\II_\kappa +\frac1{a_\kappa}\log \IP(T_{s_\mi}[0,\kappa+a_\kappa\zeta)+s_\mi \leq \kappa)}_{=:c_\kappa}\bigr]\bigr).
\end{align*}
We find convergence
$$
\lim_{\kappa\to\infty} c_\kappa= v- \sfrac{1-2\alpha}{2-2\alpha}\, (u+\zeta)^2
$$
and (as elementary calculus shows) the limit is negative by choice of $\zeta$.
\end{proof}

For $s\in\IT$ and $\kappa>0$ we denote by $\bar Z^{\ssup{s,\kappa}}=(\bar Z^{\ssup{s,\kappa}}_t)_{t\geq0}$ the random evolution given by
$$\bar Z^{\ssup{s,\kappa}}_t=\frac{Z[s,s+\kappa t]-\kappa t}{a_\kappa}.$$
Moreover, we let
$$z=(z_t)_{t\ge 0}=\Bigl(\sfrac{1-\alpha}{1-2\alpha} \bigl(t^{\frac{1-2\alpha}{1-\alpha}}\wedge 1\bigr)\Bigr)_{t\geq 0}.
$$

\begin{proof}[ of Theorem~\ref{WL}]
\emph{1st Part:} By Lemma \ref{le0924-2} the maximal indegree is related to the unimodal function~$h$ defined by
$$h(u)=-u+\sqrt {\sfrac{2-2\alpha}{1-2\alpha} u}, \qquad \mbox{ for $u\geq0$.}$$
$h$ attains its unique maximum in  $u_\ma=\frac12\frac{1-\alpha}{1-2\alpha}$ and $h(u_\ma)=u_\ma$.
We fix  $c>4\frac{1-\alpha}{1-2\alpha}$, let $\zeta=\max [h(u_\ma)- h(u_\ma\pm \eps)]$ and decompose the set 
$[0,u_\ma-\eps)\cup [u_\ma+\eps,c)$ into finitely many disjoint intervals $[u_i,v_i)$  $i\in\IIJ$, 
with mesh smaller than $\zeta/3$. Then for the hub $s^*_\kappa$ at time $\kappa>0$ one has
\begin{align}\begin{split} \label{eq1008-1}
\IP(&s^*_\kappa \in a_\kappa [u_\ma-\eps,u_\ma+\eps))\\
 &\geq \IP\Bigl(\max_{s\in a_\kappa  [u_{\ma}-\eps,u_{\ma})\cap\IT} Z[s,\kappa ] \geq \kappa + a_\kappa  (h(u_\ma)-\zeta/3)\Bigr) \\
 &\qquad \times \prod_{i\in\IIJ} \IP\Bigl(\max_{s\in a_\kappa  [u_i,v_{i})\cap\IT} Z[s,\kappa ] \leq \kappa + a_\kappa  (h(v_i)+\zeta/2)\Bigr)\\
 &\qquad \times \IP\Bigl(\max_{s\in [c\, a_\kappa,\infty)\cap\IT } Z[s,\kappa ]\leq \kappa \Bigr).
\end{split}\end{align}
By Lemma~\ref{le0924-2} the terms in the first two lines on the right terms converge to $1$.
Moreover, by Proposition~\ref{p:A_s} the third term converges to $1$, if for all sufficiently large $\kappa$ and $\kappa_+=\min [\kappa,\infty)\cap\IS$, one has  $4\vphi_{\kappa_+}\leq ca_\kappa$.  This is indeed the case, since one has $\kappa_+\leq \kappa +f(0)^{-1}$ so that by Lemma~\ref{reglema}, $4\vphi_{\kappa_+}\sim  4 \frac{1-\alpha}{1-2\alpha}\,a_\kappa$.
The statement on the size of the maximal indegree is now an immediate consequence of Lemma~\ref{le0924-2}.

\bigskip

\emph{2nd Part:} We now prove that (an appropriately scaled version of) the evolution of a hub typically lies in an open neighbourhood around $z$.\smallskip

Let  $U$ denote an open set in $\cL(0,\infty)$ that includes $z$ and denote by $U^c$ its complement  in $\cL(0,\infty)$. 
Furthermore, we set
$$A_\eps =\Bigl\{x\in\cL(0,\infty): \max_{t\in[\frac12,1]} x_t\geq 2(u_\ma-\eps)\Bigr\}$$
for $\eps\geq0$.
We start by showing that $z$ is the unique minimizer of $I$ on the set $A_0$. Indeed, applying the inverse H\"older inequality gives, 
for $x\in A_0$ with finite rate $I(x)$,
$$
I(x)\geq \sfrac 12 \int_0^1 \dot x^2_t\, t^{\frac\alpha{1-\alpha}} \,dt\geq \sfrac12 \Bigl(\int_0^1 |\dot x_t|\,dt\Bigr)^2  \Bigl(\int_0^1 t^{-\frac\alpha{1-\alpha}}\,dt\Bigr)^{-1}\geq \sfrac12 \sfrac{1-\alpha}{1-2\alpha} =u_{\max} =I(z).
$$
Moreover, one of the three inequalities is a strict inequality when $x\not=z$.
Recall that, by Lemma~\ref{AA}, $I$ has compact level sets. We first assume that one of the entries in $U^c\cap A_0$ has finite rate $I$. Since $U^c\cap A_0$ is closed, we conclude that $I$ attains its infimum on $U^c\cap A_0$. Therefore,
$$
I(U^c\cap A_0):=\inf\{I(x) \colon x\in U^c\cap A_0\}>I(z)=u_{\max}.
$$
Conversely, using again compactness of the level sets, gives
$$
\lim_{\eps\dto0} I(U^c\cap A_\eps)=I(U^c\cap A_0).
$$
Therefore, there exists $\eps>0$ such that $I(U^c\cap A_\eps)>I(z)$. Certainly, this is also true if $U^c$ contains no element of finite rate.

From the moderate deviation principle, Theorem \ref{MDP}, together with the uniformity in~$s$, see~Proposition~\ref{expoapp},
% and the exponential approximation, Lemma~\ref{le1025-1} 
we infer that
\begin{equation}\label{mdp0}
\limsup_{\kappa\to\infty} \frac1{a_\kappa} \max_{s\in\IT} \log \IP\bigl(\bar Z^{\ssup{s,\kappa}}\in U^c\cap A_\eps \bigr)\leq - I(U^c\cap A_\eps)< -I(z).
\end{equation}
It remains to show that $\IP(\bar Z^{s^*_\kappa,\kappa}\in U^c)$ converges to zero.
For $\eps>0$ and sufficiently large $\kappa$,
\begin{align*}
\IP(\bar Z^{s^*_\kappa,\kappa}\in U^c) &\leq  \IP(s^*_\kappa\not\in a_\kappa [u_{\max}-\eps,u_{\max}+\eps])\\
&\qquad + \IP \big(\max_{t\in[\frac12,1]} \bar Z^{s^*_\kappa,\kappa}_t \leq 2(u_{\max}-\eps), 
s^*_\kappa\in a_\kappa [u_{\max}-\eps,u_{\max}+\eps] \big)
 \\
&\qquad  +\IP\big(\bar Z^{s^*_\kappa,\kappa}\in U^c,\max _{t\in[\frac12,1]} 
\bar Z^{s^*_\kappa,\kappa}_t \geq 2 (u_{\max}-\eps), s^*_\kappa\in a_\kappa [u_{\max}-\eps,u_{\max}+\eps] \big).
\end{align*}
By the first part of the proof the first and second term in the last equation tend to $0$ for any $\eps>0$.
The last term can be estimated as follows
\begin{align}\begin{split}\label{eqest3}
\IP\bigl(\bar Z^{s^*_\kappa,\kappa}\in U^c,\bar Z^{s^*_\kappa,\kappa}_t \geq 2 (u_{\max}-\eps), & \, s^*_\kappa\in a_\kappa [u_{\max}-\eps,u_{\max}+\eps]\bigr)\\
& \leq \sum_{s\in \IT\cap a_\kappa [u_{\max}-\eps,u_{\max}+\eps]} \IP(\bar Z^{\ssup{s,\kappa}} \in U^c\cap A_\eps).
\end{split}
\end{align}
Moreover, $\log \# \bigl(\IT\cap a_\kappa [u_{\max}-\eps,u_{\max}+\eps]\bigr) \sim a_\kappa ( u_{\max}+\eps)$.
Since, for sufficiently small  $\eps>0$,  we have $I(U^c\cap A_\eps)>u_{\max}+\eps$ we infer from~\eqref{mdp0}
that the sum in (\ref{eqest3}) goes to zero.
\end{proof}

\begin{appendix}
\section{Appendix}

\subsection{Regularly varying attachment rules}

In the following we assume that $f\colon [0,\infty)\to (0,\infty)$ is a regularly varying attachment rule with index $\alpha<1$, and represent $f$ as $
f(u)=u^{\alpha} \ell(u)$, for $u>0$,
with a slowly varying function $\ell$.

\begin{lemma}\label{reglema} 
\begin{enumerate}
 \item 
One has
$$
\Phi(u)\sim \frac1{1-\alpha} \frac{u^{1-\alpha} }{\ell(u)}
$$
as $u$ tends to infinity and $\bar f$ admits the representation 
$$
\bar f(u)=f\circ \Phi^{-1}(u) = u^{\frac{\alpha}{1-\alpha}} \bar \ell (u),\qquad\mbox{ for } u>0,
$$
where
%$$\bar \ell(v) \sim\big((1-\alpha)  \ell(v^{1/(1-\alpha)})\big)^{\frac{\alpha}{1-\alpha}}  \ell \big((v \ell(v^{1/(1-\alpha)}))^{1/(1-\alpha)}\big)$$
$\bar \ell$ is again a slowly varying function. 

\item If additionally $\alpha<\frac12$, then $$
\vphi_u=\int_0^u \frac 1{\bar f(u)}\,du\sim \frac{1-\alpha}{1-2\alpha} \frac{u^{\frac{1-2\alpha}{1-\alpha}}}{\bar\ell(u)}.
$$
%and, for any $0\leq a<b<\infty$ and $\gamma>0$, one has 
%$$\int_{au}^{bu} \bar f(\gamma v)\,dv \sim (1-\alpha) \gamma^{\frac\alpha{1-\alpha}} (b^{\frac1{1-\alpha}}-a^{\frac1{1-\alpha}})\, u^{\beta+1} \bar %\ell(u).$$
\end{enumerate}
\end{lemma}

\begin{proof} The results follow from the theory of regularly variation, and we briefly quote the relevant results taken from \cite{BiGoTeu87}. The asymptotic formula for $\Phi$ is an immediate consequence of Karamata's theorem, Theorem 1.5.11. Moreover, by Theorem 1.5.12, the inverse of $\Phi$ is regularly varying with index $(1-\alpha)^{-1}$ so that, by Proposition 1.5.7, the composition $\bar f=f\circ \Phi^{-1}$ is regularly varying with index $\frac\alpha{1-\alpha}$. The asymptotic statement about $\vphi$ follows again by Karamata's theorem.
\end{proof}

\begin{remark}\label{r:reglem}
In the particular case where $f(u)\sim c u^\alpha$, we obtain
$$
\Phi(u)\sim \sfrac1{c(1-\alpha)} \, u^{1-\alpha}, \ \Phi^{-1}(u)\sim (c(1-\alpha)u)^{\frac 1{1-\alpha}}
$$
and
$$
\bar f(u) \sim c^{\frac{1}{1-\alpha}} \big((1-\alpha)u \big)^{\frac\alpha{1-\alpha}}.
$$
\end{remark}

\subsection{Two concentration inequalities for martingales}

\begin{lemma}\label{l:hoef}
Let $(M_n)_{n\in \IZ_+}$ be a martingale for its canonical filtration $(\cF_n)_{n\in\IZ_+}$ with $M_0=0$. We assume that there are deterministic $\sig_n\in\IR$ and $M<\infty$ such that almost surely
\begin{itemize}
\item $\var(M_{n}|\cF_{n-1})\leq \sig_n^2$ and
\item $M_{n}-M_{n-1}\leq M$.
\end{itemize}
Then, for any $\lam>0$ and $m\iN$,
$$
\IP\Bigl(\sup_{n\le m} M_n\geq \lam +\sqrt{2\sum_{n=1}^m \sig_n^2}\Bigr) \leq 2 \exp \Bigl(-\frac{\lam^2}{2(\sum_{n=1}^m \sig_n^2+M\lam/3)}\Bigr).
$$
\end{lemma}

\begin{proof}
Let $\tau$ denote the first time $n\iN$ for which $M_n \geq  \lam +\sqrt {2\sum_{n=1}^m \sig_n^2}$.
Then
\begin{align*}
\IP(M_m\geq \lam) \geq \sum_{n=1}^m \IP(\tau =n) \, \IP\Bigl(M_m-M_n\geq -\sqrt{2\sum_{i=1}^m \sig_i^2}\,\Big|\,\tau=n\Bigr).
\end{align*}
Next, observe that
$\var (M_m-M_n\,|\,\tau=n)\leq \sum_{i={n+1}}^m \sig_i^2$
so that by Chebyshev's inequality
$$
\IP\Bigl(M_m-M_n\geq -\sqrt{2\sum_{i=1}^m \sig_i^2}\,\Big|\,\tau=n\Bigr)\geq 1/2.
$$
On the other hand, a  concentration inequality of Azuma type  gives
$$
\IP(M_m\geq \lam)\leq \exp \Bigl(-\frac{\lam^2}{2(\sum_{n=1}^m \sig_n^2+M\lam/3)}\Bigr)
$$
(see for instance \cite{ChLu06}, Theorem 2.21).
Combining these estimates immediately proves the assertion of the lemma.
\end{proof}

Similarly, one can use the  classical Azuma-Hoeffding inequality to prove the following concentration inequality.

\begin{lemma}\label{conc2}
Let $(M_n)_{n\in \IZ_+}$ be a martingale such that  almost surely
$|M_{n}-M_{n-1}|\leq c_n$ for given sequence $(c_n)_{n\iN}$. Then  for any $\lam>0$ and $m\iN$
$$
\IP\Bigl(\sup_{n\le m} |M_n-M_0| \geq \lam +\sqrt{2\sum_{n=1}^m c_n^2}\Bigr) \leq 4 \exp \Bigl(-\frac{\lam^2}{2\sum_{n=1}^m c_n^2}\Bigr).
$$
\end{lemma}

\end{appendix}
\bigskip

{\bf Acknowledgments:} The first author is supported by a grant from
\emph{Deutsche Forschungsgemeinschaft (DFG)}, and the second author
enjoys the support of  the \emph{Engineering and Physical Sciences
Research Council (EPSRC)} through an Advanced Research Fellowship.
{We would also like to thank Remco~van~der~Hofstad for his
encouragement and directing our attention to the
reference~\cite{RTV07}.}
\bigskip

%\begin{thebibliography}{1}

%\end{thebibliography}
%\texttt{Needed:
%\begin{itemize}
%\item arguments to show that outdegrees are irrelevant,
%\item optimal statements for LDPs,
%\item lot of details.
%\end{itemize}}

%\texttt{Not needed:
%\begin{itemize}
%\item more models,
%\item more questions unless directly related.
%\end{itemize}}

\end{document}